\documentclass[12pt]{amsart}
\usepackage{amsmath}	
\usepackage{amssymb}
\usepackage{caption}
\usepackage{subcaption}
\usepackage{accents}
\usepackage{tikz}
\usetikzlibrary{calc,intersections,arrows,automata}

\newtheorem{theorem}{Theorem}[section]

\newtheorem{lemma}[theorem]{Lemma}
\newtheorem{proposition}[theorem]{Proposition}
\newtheorem{corollary}[theorem]{Corollary}

\theoremstyle{remark}

\newtheorem{example}[theorem]{Example}

\theoremstyle{definition}
\newtheorem{definition}[theorem]{Definition}

\numberwithin{equation}{section}

\newcommand{\ubar}[1]{\underaccent{\bar}{#1}}

\newcommand{\et}{\quad\mbox{and}\quad}

\newcommand{\bQ}{\mathbb{Q}}
\newcommand{\bR}{\mathbb{R}}
\newcommand{\bZ}{\mathbb{Z}}
\newcommand{\cC}{{\mathcal{C}}}

\newcommand{\cE}{{\mathcal{E}}}
\newcommand{\cF}{{\mathcal{F}}}
\newcommand{\cK}{{\mathcal{K}}}

\newcommand{\cO}{{\mathcal{O}}}

\newcommand{\cS}{{\mathcal{S}}}
\newcommand{\cU}{{\mathcal{U}}}

\newcommand{\disp}{\displaystyle}
\newcommand{\dist}{\mathrm{dist}}

\newcommand{\image}{\mathrm{Im}}

\newcommand{\pbot}{{\ubar{\alpha}}}
\newcommand{\ptop}{{\bar{\alpha}}}
\newcommand{\phibot}{{\ubar{\varphi}}}
\newcommand{\phitop}{{\bar{\varphi}}}
\newcommand{\psibot}{{\ubar{\psi}}}
\newcommand{\psitop}{{\bar{\psi}}}
\newcommand{\tA}{\tilde{A}}

\newcommand{\tC}{\tilde{C}}
\newcommand{\tcE}{\tilde{\cE}}

\newcommand{\tell}{\tilde{\ell}}
\newcommand{\tE}{\tilde{E}}
\newcommand{\tk}{\tilde{k}}

\newcommand{\tP}{\tilde{P}}
\newcommand{\tq}{\tilde{q}}

\newcommand{\tu}{\tilde{u}}
\newcommand{\tv}{\tilde{v}}

\newcommand{\tuP}{{\tilde\uP}}
\newcommand{\tuR}{{\tilde{\uR}}}
\newcommand{\tua}{\tilde{\ua}}

\newcommand{\ua}{\mathbf{a}}

\newcommand{\ub}{\mathbf{b}}
\newcommand{\uc}{\mathbf{c}}

\newcommand{\ue}{\mathbf{e}}

\newcommand{\uf}{\mathbf{f}}

\newcommand{\uL}{\mathbf{L}}

\newcommand{\uP}{{\mathbf{P}}}

\newcommand{\uR}{{\mathbf{R}}}
\newcommand{\uS}{{\mathbf{S}}}
\newcommand{\uu}{\mathbf{u}}
\newcommand{\uv}{\mathbf{v}}
\newcommand{\uw}{\mathbf{w}}
\newcommand{\ux}{\mathbf{x}}

\newcommand{\uy}{\mathbf{y}}

\newcommand{\uz}{\mathbf{z}}
\newcommand{\ualpha}{{\boldsymbol{\alpha}}}

\newcommand{\rien}[1]{} 

\begin{document}

\baselineskip=15.1pt

\title[On Diophantine approximation Spectra]
{On the topology of Diophantine approximation Spectra}
\author{Damien ROY}
\address{
   D\'epartement de Math\'ematiques\\
   Universit\'e d'Ottawa\\
   585 King Edward\\
   Ottawa, Ontario K1N 6N5, Canada}
\email{droy@uottawa.ca}
\thanks{Work partially supported by NSERC}

\begin{abstract}
Fix an integer $n\ge 2$.  To each non-zero point $\uu$ in $\bR^n$,
one attaches several numbers called \emph{exponents of Diophantine
approximation}.  However, as Khintchine first observed, these numbers
are not independent of each other.  This raises the problem of
describing the set of all possible values that a given family of
exponents can take by varying the point $\uu$.  To avoid
trivialities, one restricts to points $\uu$ whose coordinates
are linearly independent over $\bQ$.
The resulting set of values is called the \emph{spectrum of these
exponents}.  We show that, in an appropriate setting,
any such spectrum is a compact connected set.  In the case $n=3$, we
prove moreover that it is a semi-algebraic set closed under
component-wise minimum.
For $n=3$, we also obtain a description of the spectrum of the
exponents $(\phibot_1,\phibot_2,\phibot_3,\phitop_1,\phitop_2,\phitop_3)$
recently introduced by Schmidt and Summerer.
\end{abstract}

\maketitle

\section{Introduction}
\label{sec:intro}

The recent advances in parametric geometry of numbers have led
to an abundance of new results concerning the spectra of various
families of exponents of Diophantine approximation.  We describe
these notions below with a quick overview of the known results,
using the formalism of parametric geometry of numbers and
the notation of \cite{R2015}.  Then we
present new general properties of the spectra, mostly topological,
and discuss a particular spectrum in detail.

Fix an integer $n\ge 2$ and a non-zero point $\uu\in\bR^n$.
Then, consider the parametric family of convex bodies
of $\bR^n$ given by
\[
 \cC_\uu(q)
  =\{\ux\in\bR^n\,;\,
     \text{$\|\ux\|\le 1$ and $|\ux\cdot\uu|\le e^{-q}$}\}
  \quad (q\ge 0),
\]
where $\ux\cdot\uu$ denotes the standard scalar product of $\ux$
and $\uu$ in $\bR^n$, and $\|\ux\|=|\ux\cdot\ux|^{1/2}$
is the Euclidean norm of $\ux$.  For each $i=1,\dots,n$ and
each $q\ge 0$, set
\[
 L_{\uu,i}(q)=\log \lambda_i(\cC_\uu(q),\bZ^n)
\]
where $\lambda_i(\cC_\uu(q),\bZ^n)$ is the $i$-th minimum
of $\cC_\uu(q)$ with respect to the lattice $\bZ^n$, namely
the smallest positive real number $\lambda$ such that
$\lambda\cC_\uu(q)$ contains at least $i$ linearly independent
points of $\bZ^n$.
In 1982, using a slightly different but equivalent setting,
Schmidt noted that, for the purpose of
Diophantine approximation, it would be important
to understand the behavior of the maps
$\uL_\uu\colon[0,\infty)\to\bR^n$ given by
\[
 \uL_\uu(q)=(L_{\uu,1}(q),\dots,L_{\uu,n}(q))
 \quad (q\ge 0)
\]
(see \cite{Sc1982}).  Transposed to the present setting,
his preliminary observations can be summarized as follows.
We have $L_{\uu,1}(q)\le\cdots\le L_{\uu,n}(q)$ for each $q\ge 0$
and, by Minkowski's second convex body theorem, the function
$L_{\uu,1}(q)+\cdots+L_{\uu,n}(q)-q$ is bounded on $[0,\infty)$.
Moreover, each $L_{\uu,i}$ is a continuous piecewise linear map
with slopes $0$ and $1$.  In the same paper \cite{Sc1982}, he also made
a conjecture which was solved by Moschevitin \cite{Mo2012a}
(see also \cite{K2016}).

In \cite{SS2009, SS2013a}, Schmidt and
Summerer established further properties of the map $\uL_\uu$
which, by \cite{R2015}, completely characterize
these functions within the set of all functions from $[0,\infty)$
to $\bR^n$, modulo bounded functions.
They also introduced the quantities
\[
 \phibot_i(\uu)=\liminf_{q\to\infty} \frac{1}{q}L_{\uu,i}(q),
 \quad
 \phitop_i(\uu)=\limsup_{q\to\infty} \frac{1}{q}L_{\uu,i}(q)
 \quad (1\le i\le n).
\]
Revisiting work of Schmidt in \cite{Sc1967}, Laurent \cite{La2009b}
defined additional quantities related to
\[
 \psibot_i(\uu)=\liminf_{q\to\infty} \frac{1}{q}\sum_{j=1}^i L_{\uu,j}(q),
 \quad
 \psitop_i(\uu)=\limsup_{q\to\infty} \frac{1}{q}\sum_{j=1}^i L_{\uu,j}(q)
 \quad (1\le i< n).
\]
All of these are called \emph{exponents of Diophantine approximation}
because they appear as critical exponents in problems of Diophantine
approximation.  Of particular interest are the exponents
\[
 \phibot_1=\psibot_1, \quad
 \phitop_1=\psitop_1, \quad
 \phibot_n=1-\psitop_{n-1}
 \et
 \phitop_n=1-\psibot_{n-1}.
\]
Their \emph{spectrum} is the set of all quadruples
$(\phibot_1(\uu), \phitop_1(\uu), \phibot_n(\uu), \phitop_n(\uu))$
where $\uu$ runs through the points of $\bR^n$ whose
coordinates are linearly independent over $\bQ$.
It is easily described for $n=2$.
For $n=3$, it was determined by Laurent in \cite{La2009}.
Moreover, the spectra of the following general
families are known:
\begin{itemize}
 \item $(\phibot_1,\phitop_n)$ : the constraints come from
   Khintchine's transference principle \cite{Kh1926a, Kh1926b},
   and constructions of Jarn\'{\i}k in \cite{Ja1936a, Ja1936b}
   show that they are optimal;
 \smallskip
 \item $(\psibot_1,\psibot_2,\dots,\psibot_{n-1})$ :
   the constraints by Schmidt \cite{Sc1967} and
   Laurent \cite{La2009b} describe the full spectrum
  \cite{R2016};
 \smallskip
 \item $(\phitop_1,\phibot_n)$ : the constraints by
   Jarn\'{\i}k \cite{Ja1938} for $n=3$, and by
   German \cite{Ge2012} for $n\ge 4$ are optimal
   (see Schmidt and Summerer \cite{SS2016a}) and describe
   the full spectrum (see Jarn\'{\i}k \cite{Ja1954} for $n=3$
   and Marnat \cite{M2016} for $n\ge 4$).
\end{itemize}
For $n=4$, we also know optimal constraints on the spectrum of
$(\phibot_n,\phitop_n)$ thanks to \cite{Mo2012b}
   and \cite{SS2013b},
as well as for the spectrum of $(\phibot_1,\phitop_1)$
thanks to \cite{SS2013b}.

We propose the following notion as a general
framework to study the spectra of such families
of exponents of Diophantine approximation.

\begin{definition}
\label{intro:def:image_muT}
Let\, $T=(T_1,\dots,T_m)\colon\bR^n\to\bR^m$ be a linear map.
For each non-zero point $\uu\in\bR^n$, we define
\[
 \mu_T(\uL_\uu)
  = \Big(
    \liminf_{q\to\infty} q^{-1}T_1(\uL_\uu(q)),
    \dots,
    \liminf_{q\to\infty} q^{-1}T_m(\uL_\uu(q))
    \Big).
\]
We denote by $\image(\mu_T)$ the image of $\mu_T$, that is
the set of all $m$-tuples $\mu_T(\uu)$ where $\uu$ runs through
the non-zero points of $\bR^n$.  The \emph{spectrum} of $\mu_T$
is its subset, denoted $\image^*(\mu_T)$, consisting of
the $m$-tuples $\mu_T(\uu)$ where $\uu$ runs through
the points of $\bR^n$ with $\bQ$-linearly independent
coordinates.
\end{definition}

For example, the spectrum of the exponents
$(\phibot_1,\dots,\phibot_n,\phitop_1,\dots,\phitop_n)$,
whose definition involves both inferior and superior limits,
can be expressed as $\sigma\big(\image^*(\mu_T)\big)$
where $T\colon\bR^n\to\bR^{2n}$ and
$\sigma\colon\bR^{2n}\to\bR^{2n}$ are the linear maps
given by $T(\ux)=(\ux,-\ux)$ and $\sigma(\ux,\uy)=(\ux,-\uy)$
for any $\ux, \uy \in \bR^n$. Our first main result below
implies that it is a compact and connected subset of $\bR^{2n}$.

\begin{theorem}
\label{intro:thm:muT}
Let\, $T\colon\bR^n\to\bR^m$ be a linear map.
Then $\image^*(\mu_T)$ is
a compact connected subset of $\bR^m$.  The set $\image(\mu_T)$
is also compact but in general not connected.
\end{theorem}

As above, this implies that the spectrum of
$(\psibot_1,\dots,\psibot_{n-1},\psitop_1,\dots,\psitop_{n-1})$
is compact and connected.   The proof of the theorem
uses the following notions.

\begin{definition}
\label{intro:def:F(Lu)}
Let $\uu$ be a non-zero point of $\bR^n$. We denote by
$\cF(\uL_\uu)$ the set of all points $\ux\in\bR^n$ for which
there exists a strictly increasing unbounded sequence
of positive real numbers $(q_i)_{i\ge 1}$ such that
$q_i^{-1}\uL_\uu(q_i)$ converges to $\ux$ as $i\to\infty$.
We also denote by $\cK(\uL_\uu)$ the convex hull of
$\cF(\uL_\uu)$.
\end{definition}

So, $\cF(\uL_\uu)$ and $\cK(\uL_\uu)$ are compact subsets
of $[0,1]^n$ and,
in the notation of Theorem \ref{intro:thm:muT}, we have
\[
 \mu_T(\uL_\uu)
   =\big(\inf T_1(\cF(\uL_\uu)),\dots,\inf T_m(\cF(\uL_\uu))\big).
\]
To write this formula in a more compact way, we
use the \emph{coordinate-wise partial ordering} on $\bR^m$,
where, for any two points $\ux$ and $\uy$
in $\bR^n$, we have $\ux\le \uy$ if and only if all
coordinates of $\uy-\ux$ are $\ge 0$.
For this ordering, any bounded subset $F$ of $\bR^m$
has a greatest lower bound denoted $\inf(F)$. For each
$i=1,\dots,m$, its $i$-th coordinate is the infimum
of the set of $i$-th coordinates of the points of $F$.
Then the above formula simply becomes
\[
 \mu_T(\uL_\uu)
   =\inf T(\cF(\uL_\uu)).
\]
Moreover, the infimum of a bounded subset $S$ of $\bR^m$ is
the same as the infimum of the convex hull of $S$,
and a linear map $T\colon\bR^n\to\bR^m$ sends the convex hull
of a bounded subset $F$ of $\bR^n$ to the convex hull of
$T(F)$.  Therefore, we also have
\begin{equation}
\label{intro:eq:muT=infTK}
\mu_T(\uL_\uu)=\inf T(\cK(\uL_\uu)).
\end{equation}
In the case of dimension $n=3$, we obtain the following result.

\begin{theorem}
\label{intro:thm:u,v,w}
For each pair of points $\uv,\uw$ in $\bR^3$ having
$\bQ$-linearly independent coordinates, there exists a point
$\uu$ of $\bR^3$ which also has $\bQ$-linearly independent
coordinates, such that $\cK(\uL_\uu)$ is the
convex hull of $\cK(\uL_\uv)\cup\cK(\uL_\uw)$.
\end{theorem}

Then, for a linear map $T\colon\bR^3\to\bR^m$, we find,
using \eqref{intro:eq:muT=infTK}, that
\[
 \mu_T(\uL_\uu)
  =\inf \left(T(\cK(\uL_\uv))\cup T(\cK(\uL_\uw))\right)
  =\min \left\{\mu_T(\uL_\uv),\, \mu_T(\uL_\uw)\right\}.
\]
This gives the following result.

\begin{corollary}
Let $T\colon\bR^3\to\bR^m$ be a linear map.
For any $\ux,\uy$ in $\image^*(\mu_T)$, the point $\min\{\ux,\uy\}$
also belongs to $\image^*(\mu_T)$.
\end{corollary}

As Theorem \ref{intro:thm:muT} shows that $\image^*(\mu_T)$
is compact, it follows that this spectrum contains the infimum
of any of its subsets.  It would be interesting to know if this
property extends to the linear maps $T\colon\bR^n\to\bR^m$
with $n>3$.

In all cases where we know the complete spectrum of a family of
$m$ exponents of approximation, it appears to be a
semi-algebraic subset of $\bR^m$, that is a subset of $\bR^m$
defined by polynomial equalities and inequalities.  Here,
we show that this is true of any spectrum in dimension $n=3$.

\begin{theorem}
\label{intro:thm:semi}
For any $m\ge 1$ and any linear map $T\colon\bR^3\to\bR^m$,
the spectrum $\image^*(\mu_T)$ is a semi-algebraic subset
of $\bR^m$.
\end{theorem}

Jarn\'{\i}k showed in \cite{Ja1938} that, in dimension $n=3$,
the spectrum of $(\phibot_3,\phitop_1)$ is contained in the
arc of an algebraic curve
\[
 J=\{(x,y)\in[1/3,1/2]\times[0,1/3]\,;\, (1-2x)(1-2y)=xy\}
\]
and, in \cite{Ja1954}, that it is equal to $J$.  As mentioned
above, this remarkable result was extended by Laurent in
\cite{La2009} to a complete description of the spectrum
of $(\phibot_1,\phibot_3,\phitop_1,\phitop_3)$.  Our last main
result deals with the full family $(\phibot_1,\phibot_2,\dots,\phitop_3)$.

\begin{theorem}
\label{intro:thm:S}
Suppose that $n=3$. Then the spectrum $\cS$ of
$(\phibot_1,\phibot_2,\phibot_3,\phitop_1,\phitop_2,\phitop_3)$
is a semi-algebraic subset of $\bR^6$ defined by polynomial
inequalities with coefficients in $\bQ$ within
$\bR^2\times J\times\bR^2$.  It is moreover the topological
closure of a non-empty open subset of $\bR^2\times J\times\bR^2$.
\end{theorem}

Thus the spectrum is a closed manifold of dimension $5$ with
boundary. In Section \ref{sec:six}, we give an explicit set of inequalities
describing it, some of which have been obtained independently by
Schmidt and Summerer \cite{SS2016b} (for more details see
the comments after the statement of Theorem \ref{six:thm:ineq} below).
More precisely, we simply list half of the inequalities because,
with one exception, all others are obtained from these by a
simple transformation, in agreement with a general observation
of Schmidt and Summerer (see the remark after their
Theorem 1.2 in \cite{SS2013a} and also at the end of the
introduction of \cite{SS2013b}).  That transformation
consists in reversing inequalities and permuting the
variables representing $\phibot_i$ and $\phitop_{4-i}$
for each $i=1,2,3$.  The search for a precise formulation
and a satisfactory explanation of this duality was the initial
motivation for the present research, but it remains an open
problem.

\medskip
This paper is organized as follows.
In Section~\ref{sec:spectra}, we use results of \cite{R2015}
recalled in Section~\ref{sec:systems} to transpose the notion
of spectrum in terms of a simple class of $\bR^n$-valued
functions of one variable called $n$-systems.  In the
recent papers \cite{K2016}, \cite{M2016}, \cite{R2016},
\cite{SS2013b} and \cite{SS2016a}, the authors exhibit
points of a given spectrum by forming $n$-systems whose
graphs are invariant under a non-trivial homothety
with center at the origin.
In Section~\ref{sec:spectra}, we show that the points of a spectrum
coming from such self-similar $n$-systems all belong to a
single connected component of the spectrum.
Using tools from Sections~\ref{sec:approx},
\ref{sec:deformation} and \ref{sec:refined}, we show
moreover in Section~\ref{sec:self} that these points are dense in
the spectrum.  The latter must therefore be connected.
Its compactness is proved in Section~\ref{sec:composite}.
Thus the whole spectrum is the topological closure
of its subset of points attached to self-similar $n$-systems.
The last three sections treat the case
of dimension $n=3$ proving Theorems \ref{intro:thm:u,v,w},
\ref{intro:thm:semi} and \ref{intro:thm:S}, together
with an explicit description of the spectrum of
$(\phibot_1,\dots,\phitop_3)$.  The arguments are
geometric, taking advantage of the fact that, like
the sets $\cF(\uL_\uu)$ attached to points $\uu\in\bR^3$,
the analogous sets $\cF(\uP)$ attached to $3$-systems
$\uP$ (defined in Section~\ref{sec:spectra}) are planar sets
and therefore can easily be drawn on paper.
More precisely, the proofs are based on
a geometric description of the sets $\cF(\uP)$
given in Section~\ref{sec:3sys} for self-similar
$3$-systems $\uP$ satisfying a mild non-degeneracy
condition.

%
%

\section{Non-degenerate systems, rigid systems and canvases}
\label{sec:systems}

Fix an integer $n$ with $n\ge 2$, and let
\[
 \ue_1=(1,0,\dots,0),\dots,\ue_n=(0,\dots,0,1)
\]
denote the elements of the canonical basis of $\bR^n$.
In our setting, the notion of $(n,0)$-system
introduced and studied by Schmidt and Summerer in
\cite[\S\S 2--3]{SS2013a} takes the following form.

\begin{definition}
\label{systems:def:n-system}
Let $I$ be a subinterval of $[0,\infty)$ with
non-empty interior.  An $n$-system on $I$ is a map
$\uP=(P_1,\dots,P_n)\colon I \to \bR^n$ with the
property that, for any $q\in I$:
\begin{itemize}
\item[(S1)] $0\le P_1(q)\le\cdots\le P_n(q)$ and
  $P_1(q)+\cdots+P_n(q)=q$;
\smallskip
\item[(S2)] there exist $\epsilon>0$ and integers
  $k,\ell\in\{1,\dots,n\}$ such that
  \[
   \uP(t)=\begin{cases}
         \uP(q)+(t-q)\ue_\ell
          &\text{for any $t\in I\cap[q-\epsilon,q]$,}\\
         \uP(q)+(t-q)\ue_k
          &\text{for any $t\in I\cap[q,q+\epsilon]$;}
         \end{cases}
   \]
\item[(S3)] if $q$ is in the interior of $I$ and if the
  integers $k$ and $\ell$ from (S2) satisfy $k>\ell$,
  then $P_\ell(q)=\cdots=P_k(q)$.
\end{itemize}
We say that such a map is \emph{proper} if $P_1$ is unbounded.
\end{definition}

In \cite{SS2013a}, Schmidt and Summerer showed that the
maps $\uL_\uu\colon[0,\infty) \to\bR^n$ attached to non-zero
points $\uu$ of $\bR^n$ satisfy similar but weaker conditions
and they proposed those $n$-systems as an idealized model
for the former maps. By \cite[Theorems 8.1 and 8.2]{R2015},
the $n$-systems have the following approximation property.

\begin{theorem}
\label{systems:thm:approxLbyP}
For each non-zero point $\uu$ in $\bR^n$, there exist $q_0\ge 0$
and an $n$-system $\uP$ on $[q_0,\infty)$ such that $\uP-\uL_\uu$
is bounded on $[q_0,\infty)$.  Conversely, for any $q_0\ge 0$
and any $n$-system $\uP$ on $[q_0,\infty)$, there exists a non-zero
point $\uu\in\bR^n$ such that $\uP-\uL_\uu$ is bounded on
$[q_0,\infty)$.  The point $\uu$ has $\bQ$-linearly independent
coordinates if and only if the map $\uP$ is proper.
\end{theorem}

The last assertion follows from the fact that $\uu$ has
$\bQ$-linearly independent coordinates if and only if the function
$L_{\uu,1}\colon[0,\infty) \to \bR$ is unbounded.

We now present several constructions related to $n$-systems.

\smallskip
\noindent
\textbf{Rescaling.} If $\uP$ is an $n$-system on some
subinterval $I$ of $[0,\infty)$, then for each $\rho>0$, the map
$\tuP\colon\rho I\to \bR^n$ given by $\tuP(q)=\rho\uP(\rho^{-1} q)$
is also an $n$-system.  We say that $\uP$ is \emph{self-similar}
if $I$ is unbounded and if there exists $\rho>1$ such that
$\uP(\rho q)=\rho\uP(q)$ for each $q\in I$ (i.e.\ $\uP=\tuP$
on $\rho I$).

\smallskip
\noindent
\textbf{Gluing.}  As the definition shows,
an $n$-system $\uP\colon I\to \bR^n$ is defined by local
conditions.  In particular, it is a continuous map which admits
a left derivative $\uP'(q^-) \in \{\ue_1,\dots,\ue_n\}$
at each point $q\in I$ with $q\neq \inf(I)$, and a right
derivative $\uP'(q^+)\in \{\ue_1,\dots,\ue_n\}$ at
each point $q\in I$ with $q\neq \sup(I)$.

Suppose that we have an $n$-system $\uP$ on $[u,v]$ and an
$n$-system $\uR$ on $[v,w]$ where $0\le u<v<w$.  Let $k$
and $\ell$ be the integers for which $\uP'(v^-)=\ue_\ell$
and $\uR'(v^+)=\ue_k$.  If $\uP(v)=\uR(v)$ and $k\le \ell$,
then the map $\uS\colon [u,w]\to\bR^n$ which coincides
with $\uP$ on $[u,v]$ and with $\uR$ on $[v,w]$ is also an
$n$-system.

The above process of gluing $n$-systems can be repeated
indefinitely.  For example, suppose that $\uP(v)=\rho\uP(u)$
for some $\rho>0$.  Then, by Condition (S1), we have $\rho=v/u>1$.  Suppose
further that the integers $k$ and $\ell$ for which
$\uP'(u^+)=\ue_k$ and $\uP'(v^-)=\ue_\ell$ satisfy $k\le \ell$,
then the map $\uS\colon[u,\infty) \to\bR^n$ given by $\uS(\rho^i q)=
\rho^i\uP(q)$ for each $q\in[u,v]$ and each $i\ge 1$ is a
self-similar $n$-system which extends $\uP$.

\smallskip
\noindent
\textbf{Combined graph.}
Let $\uP=(P_1,\dots,P_n)\colon I\to \bR^n$ be an $n$-system.
Following the terminology of Schmidt and Summerer in \cite{SS2013a},
its \emph{combined graph} is the union
of the graphs of its components $P_1,\dots,P_n$ in $I\times \bR$.
Since, by Condition (S1), the map $\uP$ takes values in the set
\[
 \Delta_n=\{(x_1,\dots,x_n)\in\bR^n\,;\, x_1\le \cdots\le x_n\},
\]
this graph determines $\uP$ except in some degenerate cases.
A \emph{division number} of $\uP$
is a point of $I$ on the boundary of $I$, or an interior point
$q$ of $I$ where $\uP$ is not differentiable (i.e.\ $\uP'(q^-)\neq
\uP'(q^+)$).  The division numbers of $\uP$ form a discrete subset
of $I$.  A \emph{switch number} of $\uP$ is a point of $I$
on the boundary of $I$, or an interior point $q$ of $I$ for which
the integers $k$ and $\ell$ in Condition (S2) satisfy $k<\ell$.
Thus the switch numbers of $\uP$ are also division numbers of $\uP$.
A \emph{division point} (resp.\ \emph{switch point}) of $\uP$
is the value of $\uP$ at a division number (resp.\ at a switch number).

Suppose that $u<v$ are points of $I$ with no switch number of
$\uP$ between them, and let $k$ and $\ell$ be the integers
for which $\uP'(u^+)=\ue_k$ and $\uP'(v^-)=\ue_\ell$.  Then,
we have
\[
 P_k(u)<P_\ell(v)
 \et
 \big( P_1(u),\dots,\widehat{P_k(u)},\dots,P_n(u)\big)
 =
 \big(P_1(v),\dots,\widehat{P_\ell(v)},\dots,P_n(v)\big),
\]
where the hat on a coordinate means that it is omitted.
Moreover, the restriction of $\uP$ to $[u,v]$ is given by
\[
 \uP(q)
 =\Phi_n\big(P_1(u),\dots,\widehat{P_k(u)},\dots,P_n(u),P_k(u)+q-u\big)
 \quad
 (u\le q\le v),
\]
where $\Phi_n\colon\bR^n\to\Delta_n$ denotes
the continuous map which lists the coordinates of a point
in monotone increasing order. This formula simply expresses
the fact that the combined graph of $\uP$ over $[u,v]$ consists
of the horizontal line segments $[u,v]\times\{P_j(u)\}$ for
$j=1,\dots,\widehat{k},\dots,n$ and the line segment
of slope $1$ joining the points $(u,P_k(u))$ and $(v,P_\ell(v))$.
The $n$-system $\uP$ can be viewed as the result of gluing together
such simpler systems, for example the restrictions of $\uP$
to the closed subintervals of $I$ joining consecutive
switch numbers of $\uP$.

\smallskip
\noindent
\textbf{Canvases and non-degenerate systems.}
A \emph{non-degenerate $n$-system}
is an $n$-system $\uP$ defined on a closed subinterval of
$(0,\infty)$ whose switch points
all have $n$ distinct positive coordinates.
We will need the following related notions.

A \emph{finite pre-canvas} in
$\bR^n$ is a finite sequence of points
$(\ua^{(i)})_{1\le i \le s}$ in $\Delta_n$
of cardinality $s\ge 2$ with the property that:
\begin{itemize}
\item[(C1)] for each $i=1,\dots,s$, the coordinates
 $(a^{(i)}_1,\dots,a^{(i)}_n)$
 of $\ua^{(i)}$ form a strictly increasing sequence of
 positive real numbers;
\smallskip
\item[(C2)] for each $i=1,\dots,s-1$, there exist
integers $k_i$ and $\ell_{i+1}$ with
$1\le k_i\le \ell_{i+1}\le n$ such that
\[
 a^{(i)}_{k_i} < a^{(i+1)}_{\ell_{i+1}}
 \et
 (a^{(i)}_1,\dots,\widehat{a^{(i)}_{k_i}},\dots,a^{(i)}_n)
 =
 (a^{(i+1)}_1,\dots,\widehat{a^{(i+1)}_{\ell_{i+1}}},\dots,a^{(i+1)}_n);
\]
\item[(C3)] for each index $i$ with $1< i<s$, we have $k_i\le\ell_i$.
\end{itemize}
The pairs $(k_i,\ell_{i+1})$ with $1\le i< s$ are uniquely
determined by Condition (C2).  We call them the
\emph{transition indices} of the pre-canvas.

To a pre-canvas as above, we associate the function
$\uP\colon[q_1,q_s]\to\Delta_n$ given by
\begin{equation}
\label{eq:P}
 \uP(q)
 =\Phi_n\big(a^{(i)}_1,\dots,\widehat{a^{(i)}_{k_i}},\dots,a^{(i)}_n,
         a^{(i)}_{k_i}+q-q_i\big)
 \quad
 (1\le i< s, \ q_i\le q\le q_{i+1}),
\end{equation}
where
\begin{equation}
\label{eq:qi}
 q_i=a_1^{(i)}+\cdots+a_n^{(i)}
 \quad
 (1\le i\le s).
\end{equation}
This map is an $n$-system which satisfies $\uP(q_i)=\ua^{(i)}$
for $i=1,\dots, s$, as well as $\uP'(q_i^+)=\ue_{k_i}$
and $\uP'(q_{i+1}^-)=\ue_{\ell_{i+1}}$ for $i=1,\dots,s-1$.

A \emph{finite canvas} in
$\bR^n$ is a pre-canvas $(\ua^{(i)})_{1\le i \le s}$
which, instead of Condition (C3) satisfies the
stronger condition $k_i<\ell_i$ whenever $1<i<s$.  Then,
the numbers $q_1,\dots,q_s$ given by \eqref{eq:qi} are
the switch numbers of the associated $n$-system
$\uP\colon[q_1,q_s]\to\Delta_n$.  Since $\uP(q_i)=\ua^{(i)}$
for each $i=1,\dots,s$, that system is
non-degenerate.  We obtain in this way all
non-degenerate $n$-systems whose domain is a
compact subinterval of $(0,\infty)$.

If $(\ua^{(i)})_{1\le i \le s}$
is a finite pre-canvas, then the subsequence
obtained by deleting each point $\ua^{(i)}$ with $1<i<s$ and
$k_i=\ell_i$ is a canvas with the same associated $n$-system.
Thus, the non-degenerate $n$-systems with compact domain
are also the $n$-systems associated to the finite pre-canvases.

An \emph{infinite pre-canvas} is an infinite unbounded
sequence $(\ua^{(i)})_{1\le i}$ such that
$(\ua^{(i)})_{1\le i\le s}$ is a pre-canvas
for each integer $s\ge 2$.
Then, the corresponding
$n$-systems given by \eqref{eq:P} are the restrictions to
$[q_1,q_s]$ of a unique non-degenerate $n$-system
$\uP\colon[q_1,\infty) \to\Delta_n$.
We obtain in this way all non-degenerate $n$-systems
with unbounded domain.

The notion of an infinite canvas is similar.  The
non-degenerate $n$-systems associated to such canvases
are those with unbounded domain and infinitely many
switch numbers.  This includes all proper non-degenerate
$n$-systems.

\smallskip
\noindent
\textbf{Rigid systems.}
Let $\delta>0$. A \emph{rigid $n$-system of mesh $\delta$}
is a non-degenerate $n$-system $\uP$ whose switch points
belong to $\delta\bZ^n$. By Condition (S1), the switch
numbers of $\uP$ are then positive multiples of $\delta$.
By the above, such a system comes from a pre-canvas,
possibly infinite, contained in $\delta\bZ^n$.
By \cite[Theorems 8.1 and 8.2]{R2015}, they
have the following approximation property.

\begin{theorem}
\label{systems:thm:approxPbyRigid}
For any $\delta>0$ and any $n$-system $\uP$ with unbounded domain $I$,
there exists a rigid $n$-system $\uR\colon[q_0,\infty) \to\Delta_n$
of mesh $\delta$ with $q_0\in I$
such that $\uP-\uR$ is bounded on $[q_0,\infty)$.
\end{theorem}

\begin{example}
\label{systems:example}
For simplicity, we say that a canvas or
a pre-canvas is \emph{integral} if it is contained in $\bZ^n$.
We also say that an $n$-system is \emph{integral} if it is rigid
of mesh $1$.  The following sequence is an example of a finite
integral canvas of cardinality $9$ in $\bR^3$:
\[
 (\underline{1},2,4),\, (2,\underline{4},\overline{5}),\,
 (2,\underline{5},\overline{6}),\, (2,\underline{6},\overline{8}),\,
 (2,\underline{8},\overline{17}),\, (\underline{2},\overline{10},17),\,
 (\underline{10},\overline{13},17),\, (\underline{13},\overline{14},17),\,
 (14,17,\overline{18}).
\]
In each triple except the last, we have underlined the coordinate
which is not a coordinate of the next triple and, in the latter
we have overlined the new coordinate. In the notation of Condition (C2),
the underlined coordinate of the $i$-th triple thus has index
$k_i$ (if $i<9$) and its overlined coordinate has index
$\ell_i$ (if $i>1$).
Figure \ref{fig1} below shows the combined graph of
the corresponding $3$-system $\uP\colon[7,49]\to\Delta_3$,
with its switch numbers marked on the horizontal $q$-axis.
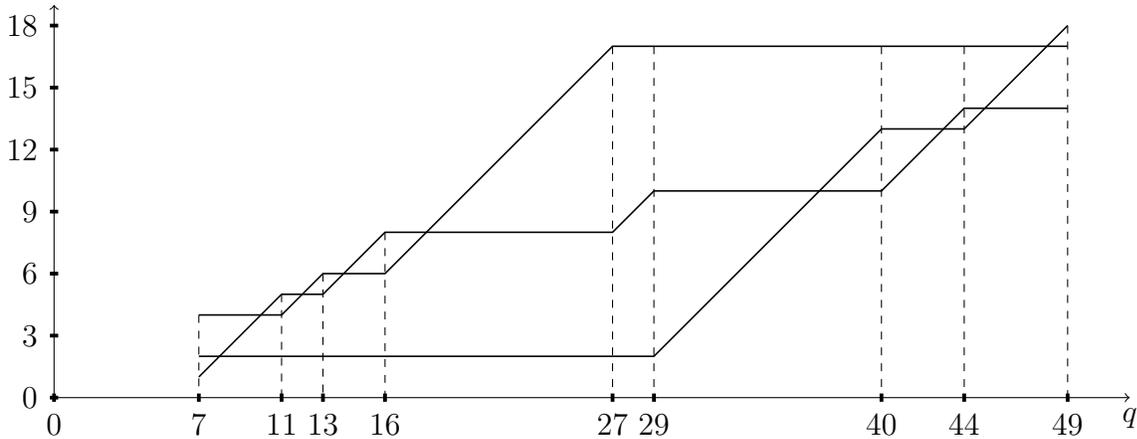
\begin{figure}[h]
   \begin{tikzpicture}[scale=0.275]
       \draw[semithick] (7,2)--(29,2)--(40,13)--(44,13)--(49,18);
       \draw[semithick] (7,1)--(11,5)--(13,5)--(16,8)--(27,8)--(29,10)--(40,10)--(44,14)--(49,14);
       \draw[semithick] (7,4)--(11,4)--(13,6)--(16,6)--(27,17)--(49,17);
       \draw[->] (-0.2,0)--(52,0) node[below]{$q$};
       \draw[->] (0,-0.2)--(0,19);
       \foreach \q in {0,7,11,13,16,27,29,40,44,49}
         \draw[line width=0.05cm] (\q,0.2)--(\q,-0.2) node[below]{$\q$};
       \foreach \q in {0,3,6,9,12,15,18}
         \draw[line width=0.05cm] (0.2,\q)--(-0.2,\q) node[left]{$\q$};
       \draw[dashed] (7,4)--(7,0);
       \draw[dashed] (11,5)--(11,0);
       \draw[dashed] (13,6)--(13,0);
       \draw[dashed] (16,8)--(16,0);
       \foreach \q in {27,29,40,44}
         \draw[dashed] (\q,17)--(\q,0);
       \draw[dashed] (49,18)--(49,0);
       \end{tikzpicture}
\caption{The combined graph of an integral $3$-system.}
\label{fig1}
\end{figure}
\end{example}

%
%

\section{Spectra of non-degenerate systems}
\label{sec:spectra}

Let $T=(T_1,\dots,T_m)\colon\bR^n\to\bR^m$ be a linear map.
For each $n$-system $\uP$ with unbounded domain,
we define $\mu_T(\uP)$ and $\cF(\uP)$ as in Definitions
\ref{intro:def:image_muT} and \ref{intro:def:F(Lu)}, upon
replacing everywhere the map $\uL_\uu$ by $\uP$.  We also
denote by $\cK(\uP)$ the convex hull of $\cF(\uP)$.
Then, the same reasoning as in Section \ref{sec:intro}
shows that
\[
 \mu_T(\uP)
   =\inf T(\cF(\uP))
   =\inf T(\cK(\uP)).
\]
We also note that, if $\uu$ is a non-zero point
of $\bR^n$ and if $\uP-\uL_\uu$ is bounded, then
we have $\cF(\uP)=\cF(\uL_\uu)$, thus
$\cK(\uP)=\cK(\uL_\uu)$ and $\mu_T(\uP)=\mu_T(\uL_\uu)$.
In view of the approximation properties stated in
Theorems \ref{systems:thm:approxLbyP}
and \ref{systems:thm:approxPbyRigid},
this yields the following alternative descriptions
of the spectrum of $\mu_T$.

\begin{theorem}
\label{spectra:thm:spectrumP}
Let $T\colon\bR^n\to\bR^m$ be any linear map, as above.
Then $\image(\mu_T)$ is the set of all points $\mu_T(\uP)$
where $\uP$ runs through the $n$-systems with unbounded
domain, while $\image^*(\mu_T)$ consists of all $\mu_T(\uP)$
where $\uP$ runs through the proper $n$-systems.  For any
$\delta>0$, this statement remains true if we restrict to
rigid $n$-systems of mesh $\delta$ of the same types.
In particular, it remains true if we restrict to
non-degenerate $n$-systems of the same types.
\end{theorem}

Thus all statements of the introduction naturally
translate into statements about $n$-systems, and we will
prove them in this form.  For that purpose, the following
fact will be useful.

\begin{proposition}
\label{spectra:prop:cvxE}
Let $\uP\colon[w_0,\infty) \to\Delta_n$ be a proper
$n$-system, let $w_0<w_1<w_2<\dots$ be its division numbers,
and let $E$ be the set of limit points of the sequence
$\big(w_i^{-1}\uP(w_i)\big)_{i\ge 1}$.  Then, $\cK(\uP)$
is the convex hull of $E$.
\end{proposition}

\begin{proof}
Since $E$ is contained in $\cF(\uP)$, its convex hull
is contained in $\cK(\uP)$.  To show the converse,
fix an arbitrary point $\ux$ of $\cF(\uP)$.  It remains
to show that $\ux$ belongs to the convex hull of $E$.
By definition, that point is given by
$\ux=\lim_{i\to\infty} t_i^{-1}\uP(t_i)$ for some
unbounded sequence $(t_i)_{i\ge 1}$ in $(w_0,\infty)$.
For each $i\ge 1$, let $u_i<v_i$ be the consecutive
division numbers of $\uP$ for which $t_i\in[u_i,v_i]$,
and put $\lambda_i=u_i(v_i-t_i)/(t_i(v_i-u_i))\in[0,1]$.
Then, we find that
\[
 t_i^{-1}\uP(t_i)
   = t_i^{-1}
     \Big( \uP(u_i) + \frac{t_i-u_i}{v_i-u_i}\big(\uP(v_i)-\uP(u_i)\big) \Big)
   = \lambda_i u_i^{-1}\uP(u_i) + (1-\lambda_i) v_i^{-1}\uP(v_i),
\]
because $\uP$ is affine on $[u_i,v_i]$.
As the triples $(u_i^{-1}\uP(u_i), v_i^{-1}\uP(v_i), \lambda_i)$
with $i\ge 1$ form a sequence in $[0,1]^n\times[0,1]^n\times[0,1]$,
and as the latter product is compact, this sequence contains
a converging subsequence whose limit is a point $(\uy,\uz,\lambda)$
in $E\times E\times [0,1]$.  By continuity, we have
$\ux=\lambda \uy + (1-\lambda) \uz$ as required.
\end{proof}

In the case where $\uP$ is self-similar, the sequence
$\big(w_i^{-1}\uP(w_i)\big)_{i\ge 1}$ is periodic and the conclusion
simplifies as follows.

\begin{corollary}
\label{spectra:cor:prop:cvxE}
Suppose moreover that $\uP$ is self-similar.
Let $\rho>1$ be such that $\uP(\rho q)=\rho\uP(q)$ for each $q\ge q_0$,
and let $s\ge 1$ be the index for which $\rho w_1=w_{s+1}$.
Then $\cK(\uP)$ is the convex hull of
$\{w_1^{-1}\uP(w_1),\dots,w_s^{-1}\uP(w_s)\}$.
\end{corollary}

The rest of this section is devoted to the following construction.

\begin{proposition}
\label{spectra:prop:P^nu}
Let $\nu\colon [0,\infty) \to [0,\infty)$ be
a strictly increasing continuous function, and let $\uP$ be
an $n$-system on a subinterval $I$ of $[0,\infty)$.
Define $\ua^\nu=(\nu(a_1),\dots,\nu(a_n))$ for
each $\ua=(a_1,\dots,a_n)\in[0,\infty)^n$.  Then, there is
a unique $n$-system, denoted $\uP^\nu$, whose image is
$\{\uP(q)^\nu\,;\,q\in I\}$.  Its switch points (resp.\ its division
points) are the points $\ua^\nu$ where $\ua$ is a switch point
(resp.\ a division point) of $\uP$.
\end{proposition}

\begin{proof}
Write $\uP=(P_1,\dots,P_n)$, and consider the function
$\sigma\colon I\to [0,\infty)$ given by
\[
 \sigma(q)=\nu(P_1(q))+\cdots+\nu(P_n(q))
 \quad
 (q\in I).
\]
This is a continuous and strictly increasing map.  So, it restricts
to a homeomorphism $\sigma\colon I\to J$ where $J$, the image
of $\sigma$, is a subinterval of $[0,\infty)$.  If there exists an
$n$-system $\uP^\nu$ as in the statement of the proposition,
then its domain is $J$ and it satisfies
\[
 \uP^\nu(\sigma(q))=\uP(q)^\nu
 \quad
 (q\in I).
\]
We now show that the function $\uP^\nu\colon J\to\bR^n$
defined by the above formula is an $n$-system.  Fix $q\in I$.
Condition (S1) clearly holds for $\uP^\nu$ at the point
$\sigma(q)\in J$.  Since $\uP$ is an $n$-system,
there are integers $k,\ell\in\{1,\dots,n\}$ such that, in a
neighborhood $\cU$ of $q$ in $I$, the component $P_j$ of $\uP$
is constant to the left of $q$ if $j\neq \ell$, and constant
to the right of $q$ if $j\neq k$.  Then $\sigma(\cU)$ is a
neighborhood $\sigma(q)$ in $J$ such that the $j$-th
component $P^\nu_j$ of $\uP^\nu$ is constant to the left
of $\sigma(q)$ if $j\neq \ell$, and constant to its right
if $j\neq k$.  Then, because of Condition (S1), the remaining component
$P^\nu_\ell$ (resp.\ $P^\nu_k$) is affine of slope $1$
to the left (resp.\ right) of $\sigma(q)$ in $\sigma(\cU)$.
Moreover, if $\ell<k$, we have $P_\ell(q)=\cdots=P_k(q)$
and so $P^\nu_\ell(\sigma(q))=\cdots=P^\nu_k(\sigma(q))$.
Thus, Conditions (S2) and (S3) hold as well, showing that $\uP^\nu$ is an
$n$-system.  Moreover, it follows from the above that
$\sigma(q)$ is a switch number or a division number of $\uP^\nu$
if and only if $q$ is of the same type for $\uP$.
\end{proof}

\begin{corollary}
\label{spectra:cor:prop:P^nu}
Let $T\colon\bR^n\to\bR^m$ be a linear map.  For each
self-similar proper $n$-system $\uP$, the point
$\mu_T(\uP)$ belongs to the connected component of
$n^{-1}T(1,\dots,1)$ in $\image^*(\mu_T)$.
\end{corollary}

\begin{proof}
Let $\uP=(P_1,\dots,P_n)\colon[w_0,\infty) \to\Delta_n$
be a proper self-similar $n$-system, let $w_0<w_1<w_2<\dots$
be its division numbers, let $\rho>1$ be such that
$\uP(\rho q)=\rho\uP(q)$ for each $q\ge w_0$, and
let $s\ge 1$ be the index for which $w_{s+1}=\rho w_1$.
For each $\lambda\in\ (0,1]$, we denote by $\uP^\lambda$,
instead of $\uP^{\nu_\lambda}$, the $n$-system given by
the above proposition for the map $\nu_\lambda$ sending
each $a\ge 0$ to $a^\lambda$.  Then, for each $q\ge w_0$,
we have
\[
 \uP^\lambda(\sigma_\lambda(q))
  = (P_1(q)^\lambda,\dots,P_n(q)^\lambda)
 \quad
 \text{where}
 \quad
 \sigma_\lambda(q)=P_1(q)^\lambda+\dots+P_n(q)^\lambda.
\]
It follows that $\uP^\lambda$ is proper and self-similar
with $\uP^\lambda(\rho^\lambda t)=\rho^\lambda\uP^\lambda(t)$
for each $t\ge \sigma_\lambda(w_0)$.  Moreover, its division
numbers are $\sigma_\lambda(w_0)<\sigma_\lambda(w_1)
<\sigma_\lambda(w_2)<\cdots$ and we have
$\sigma_\lambda(w_{s+1})=\rho^\lambda\sigma_\lambda(w_1)$.
According to Corollary \ref{spectra:cor:prop:cvxE}, this implies
that $\cK(\uP^\lambda)$ is the convex hull of the set
\[
 E_\lambda:=\{
   \sigma_\lambda(w_i)^{-1}\uP^\lambda(\sigma_\lambda(w_i))
   \,;\,
   1\le i\le s\},
\]
and so $\mu_T(\uP^\lambda)=\min T(E_\lambda)$ is a continuous
function of $\lambda$ on $(0,1]$.  Since $\uP$ is proper and
self-similar, we also note that $P_1$ vanishes nowhere on
$[w_0,\infty)$ except possibly at $w_0$ if $w_0=0$.  Thus,
as $\lambda$ goes to $0$ in $(0,1]$, the ratios
$P_j(w_i)^\lambda/P_1(w_i)^\lambda$ with $1\le i\le s$ and
$1\le j\le n$ all converge to $1$.  Consequently, the
elements of $E_\lambda$ converge to $n^{-1}\ue$ where
$\ue=(1,\dots,1)$, and so $\mu_T(\uP^\lambda)$ converges
to $n^{-1}T(\ue)$.  Moreover, the proper $n$-system $\uS$
attached to the canvas $\{(i+1,\dots,i+n)\,;\,i\ge 0\}$
has $\cF(\uS)=\{n^{-1}\ue\}$ and so $\mu_T(\uS)=n^{-1}T(\ue)$
belongs to $\image^*(\mu_T)$.  This means that the points
$\mu_T(\uP)$ and $\mu_T(\uS)$
are connected by an arc in $\image^*(\mu_T)$.  Thus they
belong to the same connected component of that set.
\end{proof}

\section{Approximation of sets}
\label{sec:approx}

For any subsets $E$ and $F$ of $\bR^n$, we define
\[
 E+F = \{\ux+\uy\,;\, \ux\in E \text{ and } \uy\in F\}.
\]
When $E$ and $F$ are bounded, we define their mutual
distance $\dist(E,F)$ to be the infimum of all $\epsilon >0$
such that
\[
 E\subseteq F+[-\epsilon,\epsilon]^n
 \et
 F\subseteq E+[-\epsilon,\epsilon]^n.
\]
This function satisfies the triangle inequality:
\[
 \dist(E,G)\le \dist(E,F)+\dist(F,G)
\]
for any bounded subsets $E$, $F$ and $G$ of $\bR^n$.
It turns the set of all compact
subsets of $[0,1]^n$ into a complete metric space.

The goal of this section is the following
approximation result.

\begin{lemma}
\label{approx:lemma:selection}
Let $\uR$ be an $n$-system on some unbounded
closed subinterval $[q_0,\infty)$ of $(0,\infty)$,
and let $\epsilon>0$. Then, there exists $u\ge q_0$
such that
\begin{equation}
\label{approx:lemma:selection:eq1}
 \{q^{-1}\uR(q)\,;\, q\ge u\} \subseteq
  \cF(\uR) + [-\epsilon,\epsilon]^n.
\end{equation}
Moreover, for each $u\ge q_0$, there exists $v>u$
such that
\begin{equation}
\label{approx:lemma:selection:eq2}
 \cF(\uR) \subseteq
   \{q^{-1}\uR(q)\,;\, u\le q\le v\}
    + [-\epsilon,\epsilon]^n.
\end{equation}
\end{lemma}

Note that, if $u=u_0$ satisfies \eqref{approx:lemma:selection:eq1},
then any choice of $u$ with $u\ge u_0$ does so.  Similarly,
if for a given $u\ge q_0$, the choice of $v=v_0$ satisfies
\eqref{approx:lemma:selection:eq2}, then any choice of $v$ with
$v\ge v_0$ does so.

\begin{proof}
Since $\cF(\uR)\subseteq [0,1]^n$, there exist finitely
many points $\ux_1,\dots,\ux_N \in \cF(\uR)$ such that
\[
 \cF(\uR)
 \subseteq
 \cO:=\{\ux_1,\dots,\ux_N\} + (-\epsilon/2,\epsilon/2)^n.
\]
Since $\cO$ is an open set, the difference $[0,1]^n\setminus\cO$
is compact.  By definition of $\cF(\uR)$, there is no
unbounded sequence of real numbers $q_0<q_1<q_2<\cdots$
such that $q_i^{-1}\uR(q_i)\in [0,1]^n\setminus\cO$
for all $i\ge 1$ (otherwise a subsequence of
$(q_i^{-1}\uR(q_i))_{i\ge 1}$ would converge to a point of
$\cF(\uR)$ in $[0,1]^n\setminus\cO$, and there is no such
point).  Thus there exists $u\ge q_0$ such that
\[
 \{q^{-1}\uR(q)\,;\,q\ge u\}
  \subseteq \cO
  \subseteq \cF(\uR)+ (-\epsilon/2,\epsilon/2)^n.
\]
This proves the first assertion of the lemma.  Finally,
let $u\ge q_0$ be arbitrary.  For each $j=1,\dots,N$,
there exists $q_j>u$ such that
$\|\ux_j-q_j^{-1}\uR(q_j)\|_\infty < \epsilon/2$.
Then, we have
\[
 \cF(\uR)
  \subseteq
  \cO
  \subseteq
  \{q_j^{-1}\uR(q_j)\,;\,1\le j\le N\}
    +  (-\epsilon,\epsilon)^n,
\]
and so \eqref{approx:lemma:selection:eq2} holds with
$v=\max\{q_1,\dots,q_N\}$.
\end{proof}

\begin{corollary}
\label{approx:cor:lemma:selection}
When $u$ and $v$ satisfy both conditions of
the lemma, we have
\[
 \dist(\cF(\uR),F)\le \epsilon
 \quad\text{where}\quad
F=\{q^{-1}\uR(q)\,;\,u\le q\le v\}.
\]
\end{corollary}

\section{Two types of deformations}
\label{sec:deformation}

Our next goal is to show that, for each linear map
$T\colon \bR^n\to\bR^m$, the associated spectrum
$\image^*(\mu_T)$ is a compact subset of $\bR^m$, and that
a dense subset of that set is provided by the points
$\mu_T(\uS)$ where $\uS$ runs through all self-similar
rigid $n$-systems of a given mesh.  This requires
constructing new $n$-systems from others
by restricting them to suitable compact intervals,
and then by modifying and rescaling the resulting
pieces in order to glue them together.  To select
suitable intervals, we use Lemma \ref{approx:lemma:selection}
and its corollary from the previous section.  Then, we make
two types of deformations on the pieces.  The first one
is provided by the following result which allows us
to modify a rigid $n$-system on a compact interval near
the end of that interval.

\begin{proposition}
\label{deformation:prop:extension}
Let $\delta>0$,
let $\uP=(P_1,\dots,P_n)\colon [u,v]\to\Delta_n$
be a rigid $n$-system of mesh $\delta$ defined
on a compact subinterval $[u,v]$ of $(0,\infty)$
and let $\uc=(c_1,\dots,c_n)$ be a strictly
increasing sequence of positive multiples of $\delta$.
Suppose that $P_j(v)\le c_j$ for
$j=1,\dots,n$, and let $w=c_1+\cdots+c_n$.
Then there exist a rigid $n$-system
$\tuP=(\tP_1,\dots,\tP_n)\colon [u,w]\to\Delta_n$
of the same mesh $\delta$ and a surjective
map $A\colon [u,w]\to [u,v]$ with the
following properties:
\begin{itemize}
 \item[(i)] $\tuP(q)=\uP(q)$ and $A(q)=q$
   for each $q\in[u,v]$ such that $P_n(q)<P_n(v)$;
 \item[(ii)] $\tuP(w)=\uc$;
 \item[(iii)] $0\le \tP_j(q)-P_j(A(q)) \le c_j-P_j(v)$
   for each $q\in[u,w]$ and each $j=1,\dots,n$.
\end{itemize}
\end{proposition}

\begin{proof}
Suppose first that there exists an index $m$
with $1\le m\le n$ such that $c_m>P_m(v)$
while $c_j=P_j(v)$ for any other index $j$
in the same range. Let $(\ua^{(1)},\dots,\ua^{(s)})$
be the finite canvas attached to $\uP$,
and let $(q_1,\dots,q_s)$ be the corresponding
sequence of switch numbers, so that $u=q_1$,
$v=q_s$ and $\uP(q_i)=\ua^{(i)}$ for $i=1,\dots,s$.
For $i=1,\dots, s-1$, we also denote by
$(k_i,\ell_{i+1})$ the pair of transition indices
characterized by
$P'_{k_i}(q_i^+) = P'_{\ell_{i+1}}(q_{i+1}^-)=1$.
Define
\begin{equation}
 \label{deformation:prop:extension:eq1}
 \delta_m=c_m-P_m(v),
\end{equation}
and denote by $r$ the smallest index with
$1\le r\le s$ such that
\[
 a_{j}^{(r)}=\cdots=a_{j}^{(s)}
 \quad
 \text{for $j=m,\dots,n$.}
\]
Then we have $\ell_i<m$ for each $i$ with
$r<i\le s$, and moreover $\ell_r\ge m$
if $r\ge 2$.  In the case where $r\ge 2$ and $\ell_r>m$,
we form the canvas
\[
 (\tua^{(1)},\dots,\tua^{(s+1)})
 = (\ua^{(1)},\dots,\ua^{(r)},\tua^{(r+1)},\dots,\tua^{(s+1)})
\]
where $\tua^{(r+1)},\dots,\tua^{(s+1)}$ are
obtained respectively from $\ua^{(r)},\dots,\ua^{(s)}$
by replacing their $m$-th coordinate by $c_m$.
Its associated transition indices and switch numbers
are given by
\[
 (\tk_i,\tell_{i+1})
   =\begin{cases}
     (k_i,\ell_{i+1}) &\text{if $1\le i \le r-1$,}\\
     (m,m) &\text{if $i=r$,}\\
     (k_{i-1},\ell_i) &\text{if $r+1\le i \le s+1$,}
    \end{cases}
 \quad
 \tq_i
   =\begin{cases}
      q_i &\text{if $1\le i \le r$,}\\
      q_{i-1}+\delta_m &\text{if $r+1\le i\le s+1$.}
    \end{cases}
\]
In the complementary cases where $r=1$ and where $r\ge 2$
with $\ell_r=m$, we construct a canvas
$(\tua^{(1)},\dots,\tua^{(s)})$ directly from
$(\ua^{(1)},\dots,\ua^{(s)})$ by replacing
the $m$-th coordinate $a_m^{(i)}$ of $\ua^{(i)}$
by $c_m$ for $i=r,\dots,s$.  Its transition indices
remain the same as those of the original canvas,
and its switch numbers are $\tq_i=q_i$ for
$i=1,\dots,r-1$ and $\tq_i=q_i+\delta_m$ for
$i=r,\dots,s$.  In all cases, the corresponding
$n$-system $\tuP=(\tP_1,\dots,\tP_n)
\colon [q_1,q_s+\delta_m]\to \Delta_n$
satisfies
\[
 \tP_j(q)
 =\begin{cases}
    P_j(q)
     &\text{if $q_1\le q\le q_r$,}\\
    P_j(q_r)
     &\text{if $q_r\le q\le q_r+\delta_m$ and $j\neq m$,}\\
    P_m(q_r)+q-q_r
     &\text{if $q_r\le q\le q_r+\delta_m$ and $j=m$,}\\
    P_j(q-\delta_m)
     &\text{if $q_r+\delta_m\le q \le q_s+\delta_m$ and $j\neq m$,}\\
    P_m(q-\delta_m)+\delta_m = c_m
     &\text{if $q_r+\delta_m\le q \le q_s+\delta_m$ and $j=m$.}
  \end{cases}
\]
So, it satisfies the conditions (i) to (iii) for
the choice of $w=q_s+\delta_m$ and of the map
$A\colon [q_1,q_s+\delta_m]\to [q_1,q_s]$ given by
\[
 A(q)
  =\begin{cases}
     q &\text{if $q_1\le q\le q_r$,}\\
     q_r &\text{if $q_r\le q\le q_r+\delta_m$,}\\
     q-\delta_m &\text{if $q_r+\delta_m\le q \le q_s+\delta_m$.}
  \end{cases}
\]
This proves the proposition in the case where $\uc$
and $\uP(v)$ differ by at most one coordinate, because,
if $\uc=\uP(v)$, it suffices to choose $w=v$,
$\tuP=\uP$ and to take for $A$ the identity
map of $[u,v]$.  Moreover, for each $q\in [u,v]$ with
$P_n(q)< P_n(v)$, the constructed map $\tuP$
satisfies $\tP_n(q)<\tP_n(w)$, because $\tP_n(q)=P_n(q)$
and $P_n(v) \le c_n=\tP_n(w)$.

For the general case, define
\[
 \uc^{(m)}=(P_1(v),\dots,P_{m-1}(v),c_m,\dots,c_n)
 \et
 v^{(m)}=v+\delta_m+\cdots+\delta_n,
\]
for $m=1,\dots,n+1$, with $\delta_m$ given
by \eqref{deformation:prop:extension:eq1} as above.
Starting from the $n$-system $\uP^{(n+1)}=\uP$
on $[u,v^{(n+1)}]=[u,v]$, the special case
proved above allows us to construct recursively,
for each $m=n,\dots,1$, an $n$-system $\uP^{(m)}$
on $[u,v^{(m)}]$ and a surjective
map $A^{(m)}\colon [u,v^{(m)}]\to [u,v^{(m+1)}]$
such that
\begin{itemize}
 \item[1)] $\uP^{(m)}(q)=\uP(q)$ and $A^{(m)}(q)=q$
   for each $q\in[u,v]$ such that $P_n(q)<P_n(v)$,
 \item[2)] $\uP^{(m)}(v^{(m)})=\uc^{(m)}$,
 \item[3)] $P^{(m)}_j(q)=P^{(m+1)}_j(A^{(m)}(q))$
   for each $q\in[u,v^{(m)}]$ and each $j=1,\dots,n$
   with $j\neq m$,
 \item[4)] $0\le P^{(m)}_m(q)-P^{(m+1)}_m(A^{(m)}(q))
             \le \delta_m$ for each $q\in[u,v^{(m)}]$.
\end{itemize}
Then the composite map
$A=A^{(n)}\circ\cdots\circ A^{(1)}\colon
[u,v^{(1)}]\to [u,v]$ is surjective.  For that map
and for the choice of $w=v^{(1)}$, the $n$-system
$\tuP=\uP^{(1)}$ satisfies all conditions
(i) to (iii).
\end{proof}

\begin{example}
\label{deformation:example}
To illustrate the above construction, suppose that
$\uP\colon[25, 42]\to\Delta_4$ is the $4$-system
associated with the integral canvas
\[
 (1,3,\underline{9},12), \,
 (1,\underline{3},12,\overline{15}), \,
 (\underline{1},\overline{6},12,15), \,
 (6,\overline{9},12,15)
\]
and that $\uc=(8,12,16,20)$ (using the same
convention as in Example \ref{systems:example}
for underlining and overlining coordinates
of points in a canvas).
Then starting with $\uP^{(5)}=\uP$, the canvases
associated to $\uP^{(4)},\dots,\uP^{(1)}$ are
respectively
\[
 \begin{aligned}
 \uP^{(4)}\colon &(1,3,\underline{9},12), \,
 (1,\underline{3},12,\overline{20}), \,
 (\underline{1},\overline{6},12,20), \,
 (6,\overline{9},12,20),\\
 \uP^{(3)}\colon &(1,3,\underline{9},12), \,
 (1,3,\underline{12},\overline{20}), \,
 (1,\underline{3},\overline{16},20), \,
 (\underline{1},\overline{6},16,20), \,
 (6,\overline{9},16,20),\\
 \uP^{(2)}\colon &(1,3,\underline{9},12), \,
 (1,3,\underline{12},\overline{20}), \,
 (1,\underline{3},\overline{16},20), \,
 (\underline{1},\overline{6},16,20), \,
 (6,\overline{12},16,20),\\
 \uP^{(1)}\colon &(1,3,\underline{9},12), \,
 (1,3,\underline{12},\overline{20}), \,
 (1,\underline{3},\overline{16},20), \,
 (\underline{1},\overline{6},16,20), \,
 (\underline{6},\overline{12},16,20), \,
 (\overline{8},12,16,20).
 \end{aligned}
\]
\end{example}

In practice, we will use the proposition in the
following form.

\begin{corollary}
\label{deformation:cor:prop:extension}
With the notation and hypotheses of Proposition
\ref{deformation:prop:extension}, suppose that $v>nu$, and
choose $\epsilon>0$ such that $c_j\le P_j(v) + \epsilon v$
for $j=1,\dots,n$. Then, we have
\begin{itemize}
 \item[(i)] $\tuP(q)=\uP(q)$ for each $q\in [u,v/n)$,
 \item[(ii)] $\tuP(w)=\uc$,
 \item[(iii)] $\dist(E,\tE)\le n(n+1)\epsilon$,
\end{itemize}
where $E=\{q^{-1}\uP(q)\,;\, u\le q\le v\}$
and $\tE=\{q^{-1}\tuP(q)\,;\, u\le q\le w\}$.
\end{corollary}

\begin{proof}
For each $q\in[u,v/n)$, we have
\[
 P_n(v) \ge \frac{1}{n}\sum_{j=1}^n P_j(v) = \frac{v}{n}
   > q = \sum_{j=1}^n P_j(q) \ge P_n(q),
\]
and so $\tuP(q)=\uP(q)$ and $A(q)=q$ by Proposition
\ref{deformation:prop:extension} (i).
This proves part (i) of the corollary.
Part (ii) needs no proof as it is
the same as Proposition \ref{deformation:prop:extension} (ii).
For each $q\in [v/n,w]$, we find,
by Proposition \ref{deformation:prop:extension} (iii), that
\[
 0\le q-A(q)
    = \sum_{j=1}^n \big(\tP_j(q)-P_j(A(q))\big)
    \le \sum_{j=1}^n (c_j-P_j(v))
    \le n \epsilon v
    \le n^2\epsilon q,
\]
and also $\|\tuP(q)-\uP(A(q))\|_\infty
\le \epsilon v\le n\epsilon q$.  Then,
using $\|\uP(A(q))\|_\infty\le A(q)$, we find that
\[
 \begin{aligned}
 \|q^{-1}\tuP(q)-A(q)^{-1}\uP(A(q))\|_\infty
   &\le q^{-1}\|\tuP(q)-\uP(A(q))\|_\infty
     + |q^{-1}-A(q)^{-1}|\,\|\uP(A(q))\|_\infty\\
   &\le n\epsilon + q^{-1}|A(q)-q|\\
   &\le n(n+1)\epsilon
 \end{aligned}
\]
for each $q\in [v/n,w]$ and therefore also
for each $q\in [u,w]$ (because $\tuP(q)=\uP(q)$ and $A(q)=q$
when $q\in [u,v/n)$).  This yields part (iii) of
the corollary.
\end{proof}

The other type of deformation that we need is provided
by the next result, allowing us to increase by a given
constant $b$ the components of a rigid $n$-system.

\begin{lemma}
\label{deformation:lemma:RP}
Let $\uR=(R_1,\dots,R_n)\colon [u,v]\to\Delta_n$
be a rigid $n$-system of mesh $\delta$ defined
on a compact subinterval $[u,v]$ of $(0,\infty)$,
and let $b\ge 0$ be a multiple of $\delta$.  Then
the map $\uP\colon [u+nb,v+nb]\to\Delta _n$ given by
\[
 \uP(q)=(b+R_1(q-nb),\dots,b+R_n(q-nb))
\]
for each $q\in [u+nb,v+nb]$ is also a rigid $n$-system
of mesh $\delta$.  Moreover, we have
\[
 \|\,(q+nb)^{-1}\uP(q+nb)-q^{-1}\uR(q)\,\|_\infty
 \le \frac{(n+1)b}{u+nb}\,,
\]
for each $q\in [u,v]$.
\end{lemma}

\begin{proof} Put $\ub=(b,\dots,b)\in \Delta_n$,
and let $(\ua^{(i)})_{1\le i\le s}$ be the canvas attached to
$\uR$.  Then $(\ub+\ua^{(i)})_{1\le i\le s}$ is a canvas
of the same mesh $\delta$, with the same pairs of transition
indices, and $\uP$ is the corresponding $n$-system.
This proves the first assertion.  For the second one,
fix $q\in [u,v]$.  Using $\|\uR(q)\|_\infty\le q$, we find that
\[
 \begin{aligned}
 \|(q+nb)^{-1}\uP(q+nb)-q^{-1}\uR(q)\|_\infty
   &= \|(q+nb)^{-1}\ub + ((q+nb)^{-1}-q^{-1})\uR(q)\|_\infty\\
   &\le (q+nb)^{-1}b + |q(q+nb)^{-1}-1|
   = \frac{(n+1)b}{q+nb}.
   \qedhere
 \end{aligned}
\]
\end{proof}

%
%

\section{A refined approximation result}
\label{sec:refined}

Recall that an $n$-system $\uP=(P_1,\dots,P_n)$ is proper
if its first component $P_1$ is unbounded.  This implies
that the domain $I$ of $\uP$ is unbounded and that there are
arbitrarily large values of $q$ in $I$ for which $\uP'(q^+)=\ue_1$
(otherwise $P_1$ would be eventually constant, against the
hypothesis).  Then, we may define a subset $\cF(\uP,\ue_1)$
of $\cF(\uP)$ as follows.

\begin{definition}
\label{refined:def:F(P,e1)}
Let $\uP=(P_1,\dots,P_n)$ be a proper $n$-system, and let
$I\subseteq [0,\infty)$ be its domain. We denote by
$\cF(\uP,\ue_1)$ the set of all points $\ux\in\bR^n$ for which
there exists a strictly increasing unbounded sequence
of positive real numbers $(q_i)_{i\ge 1}$ in $I$ such that
$\uP'(q_i^+)=\ue_1$ for each $i\ge 1$, and $\lim_{i\to\infty}
q_i^{-1}\uP(q_i)=\ux$.
\end{definition}

Equivalently, $\cF(\uP,\ue_1)$ is the set of accumulation points
of the ratios $q^{-1}\uP(q)\in [0,1]^n$ where $q\in I$
with $\uP'(q^+)=\ue_1$. It is not empty since $[0,1]^n$ is
compact.

We can now state the main result of this section
which refines Corollary \ref{approx:cor:lemma:selection} using
the deformation process of Lemma \ref{deformation:lemma:RP}.
For the sake of compactness, we define
\begin{equation}
\label{refined:eq:ue}
 \ue=\ue_1+\cdots+\ue_n=(1,\dots,1).
\end{equation}

\begin{proposition}
\label{refined:prop:ajustement}
Let $\delta>0$, let $\uR$ be a proper rigid $n$-system of
mesh $\delta$ and let $\ux\in\cF(\uR,\ue_1)$.
For any $\epsilon_1>0$ and any $\epsilon_2>0$,
there exist positive multiples $u$ and $v$
of $\delta$ with $u<v$ and a rigid $n$-system
$\uP\colon [u,v]\to\Delta_n$ of mesh
$\delta$ satisfying the following properties:
\begin{itemize}
 \item[(i)] $\uP'(u^+)=\ue_1$,
 \smallskip
 \item[(ii)] $(1+3n\epsilon_1)^{-1}(\ux+\epsilon_1\ue)
   \le u^{-1}\uP(u) \le \ux+4\epsilon_1\ue$,
 \smallskip
 \item[(iii)] $\ux-2\epsilon_2\ue
   \le v^{-1}\uP(v) \le \ux+2\epsilon_2\ue$,
 \medskip
 \item[(iv)] $\dist(\cF(\uR), E) \le 4(n+1)\epsilon_1$
   \ where \  $E=\{q^{-1}\uP(q)\,;\,q\in [u,v]\}$.
\end{itemize}
Moreover, we may take $u$ to be arbitrarily large and,
for a given appropriate choice of $u$, we may take $v$ to be
arbitrarily large.
\end{proposition}

\begin{proof}
Fix a choice of $\epsilon_1,\epsilon_2>0$.  By definition of
$\cF(\uR,\ue_1)$, there exist arbitrarily large multiples
$u_0$ and $v_0$ of $\delta$ in the domain of $\uR$
satisfying
\begin{equation}
 \label{refined:prop:ajustement:eq1}
 \uR'(u_0^+)=\ue_1,
 \quad
 \|u_0^{-1}\uR(u_0)-\ux\|_\infty \le \epsilon_1
 \et
 \|v_0^{-1}\uR(v_0)-\ux\|_\infty \le \epsilon_2.
\end{equation}
By Corollary \ref{approx:cor:lemma:selection}, we can
choose them so that $u_0<v_0$ and
\begin{equation}
 \label{refined:prop:ajustement:eq2}
 \dist(\cF(\uR),F) \le \epsilon_1
 \quad\text{where}\quad
 F=\{q^{-1}\uR(q)\,;\, u_0\le q\le v_0\}.
\end{equation}
More precisely, we can take for $u_0$ any sufficiently
large multiple of $\delta$ satisfying the first two
conditions in \eqref{refined:prop:ajustement:eq1} and, once
$u_0$ is fixed, we can take for $v_0$ any sufficiently
large multiple of $\delta$ satisfying the last
condition in \eqref{refined:prop:ajustement:eq1}.  Put
\[
 b=\lceil 2\epsilon_1 u_0\delta^{-1}\rceil\delta,
 \quad
 u=u_0+nb
 \et
 v=v_0+nb,
\]
and consider the map $\uP\colon [u,v]\to\Delta_n$ given by
$\uP(q)=\uR(q-nb)+b\/\ue$ for each $q\in[u,v]$. According
to Lemma \ref{deformation:lemma:RP}, this is a rigid
$n$-system of mesh $\delta$ such that
\begin{equation}
 \label{refined:prop:ajustement:eq3}
 \dist(F,E)\le \frac{(n+1)b}{u_0+nb}
 \quad\text{where}\quad
 E=\{q^{-1}\uP(q)\,;\,q\in [u,v]\}.
\end{equation}
Assuming $u_0$ large enough, we have $2\epsilon_1u_0\le b\le
3\epsilon_1u_0$.  Since $\|u_0^{-1}\uR(u_0)-\ux\|_\infty
\le \epsilon_1$, this yields
\begin{align*}
 u^{-1}\uP(u)
 &\ge(u_0+3n\epsilon_1u_0)^{-1}(\uR(u_0)+2\epsilon_1u_0\ue)
 \ge (1+3n\epsilon_1)^{-1}(\ux+\epsilon_1\ue), \\
 u^{-1}\uP(u)
 &\le u_0^{-1}(\uR(u_0)+3\epsilon_1u_0\ue)
 \le \ux+4\epsilon_1\ue.
\end{align*}
Thus, $\uP$ satisfies condition (ii).  It also satisfies
condition (i) since $\uP'(u^+)=\uR'(u_0^+)=\ue_1$.
Assuming $v_0/u_0$ large enough, we also find that
\[
 \|v^{-1}\uP(v)-v_0^{-1}\uR(v_0)\|_\infty
 = \|(v_0+nb)^{-1}(\uR(v_0)+b\/\ue)-v_0^{-1}\uR(v_0)\|_\infty
 \le \epsilon_2.
\]
Then condition (iii) follows as well,
using \eqref{refined:prop:ajustement:eq1}.
Finally, \eqref{refined:prop:ajustement:eq2}
and \eqref{refined:prop:ajustement:eq3} yield
\[
 \dist(\cF(\uR),E)
   \le \epsilon_1+\frac{(n+1)b}{u_0+nb}
   \le \epsilon_1+\frac{3(n+1)\epsilon_1u_0}{u_0}
   \le 4(n+1)\epsilon_1
\]
as requested in condition (iv).  By varying $u_0$, we can make $u$
arbitrarily large and, for a fixed $u_0$, we can make $v$ arbitrarily
large by varying $v_0$.
\end{proof}

%
%

\section{Approximation by self-similar systems}
\label{sec:self}

The next result corroborates the importance of self-similar
systems.  We use it below together with Corollary
\ref{spectra:cor:prop:P^nu} to conclude that any spectrum
of exponents of approximation (in the sense of
Definition \ref{intro:def:image_muT}) is connected.

\begin{theorem}
\label{self:thm}
Let $\uR$ be a proper $n$-system.  For any
$\delta>0$ and $\epsilon>0$,
there is a self-similar rigid $n$-system
$\uS$ of mesh $\delta$ with
$\dist\big(\cF(\uR), \cF(\uS)\big) \le \epsilon$.
\end{theorem}

\begin{proof}
Fix $\delta>0$ and $\epsilon>0$.
By Theorem \ref{systems:thm:approxPbyRigid}, we may assume
that $\uR$ is rigid of mesh $\delta$.  Choose
$\ux\in\cF(\uR,\ue_1)$ and consider the $n$-system
$\uP\colon [u,v]\to\Delta_n$ given
by Proposition \ref{refined:prop:ajustement}
for the choice of $\epsilon_1=\min\{1,\epsilon\}/(16n^2(n+1))$ and
$\epsilon_2=\epsilon_1/2$.  We take $v/u$ large enough
with $v/u>n$, so that the integer
$m=\lceil(1+3n\epsilon_1)vu^{-1}\rceil$ satisfies
$m\le (1+4n\epsilon_1)vu^{-1}$.  This yields
\begin{align*}
m\uP(u)
  &\ge (1+3n\epsilon_1)vu^{-1}\uP(u)
   \ge v(\ux+\epsilon_1\ue)
    = v(\ux+2\epsilon_2\ue)
   \ge \uP(v),\\
m\uP(u)
  &\le mu(\ux+4\epsilon_1\ue)
   \le v(1+4n\epsilon_1)(\ux+4\epsilon_1\ue)\\
  &\le v(\ux+4\epsilon_1\ue+4n\epsilon_1(1+4\epsilon_1)\ue)
   \le v(\ux+4\epsilon_1\ue+5n\epsilon_1\ue)
   \le \uP(v) + 8n\epsilon_1v\ue,
\end{align*}
using $\ux\le\ue$, $\epsilon_1\le 1/16$ and $n\ge 2$,
where $\ue$ is given by \eqref{refined:eq:ue}.
Then, Corollary \ref{deformation:cor:prop:extension}
applies to $\uP$ for the choice of $\uc=m\uP(u)$.
It provides a rigid $n$-system $\tuP\colon [u,mu] \to \Delta_n$
of mesh $\delta$ such that
\begin{flalign}
 \label{self:thm:eq7}
 \quad
 &\tuP(u)=\uP(u) \et
   \tuP'(u^+)=\uP'(u^+)=\ue_1, &\\
 \label{self:thm:eq8}
 &\tuP(mu) = m\uP(u),\\
 \label{self:thm:eq9}
 &\dist(E,\tE)\le 8n^2(n+1)\epsilon_1\le \epsilon/2,
\end{flalign}
where $E=\{q^{-1}\uP(q)\,;\, u\le q\le v\}$
and $\tE=\{q^{-1}\tuP(q)\,;\, u\le q\le mu\}$.
By \eqref{self:thm:eq7} and \eqref{self:thm:eq8},
the map $\tuP$ extends to a self-similar rigid $n$-system
$\uS\colon[u,\infty) \to\Delta_n$ of mesh $\delta$ satisfying
$\uS(mq)=m\uS(q)$ for each $q\ge u$, and thus
$\cF(\uS)=\tE$.  Since we have $\dist(\cF(\uR),E)
\le 4(n+1)\epsilon_1$ by the choice of $\uP$, we conclude
from \eqref{self:thm:eq9} that
$\dist\big(\cF(\uR),\cF(\uS)\big)\le 4(n+1)\epsilon_1+\epsilon/2
\le \epsilon$.
\end{proof}

\begin{corollary}
\label{self:cor}
Let $T\colon\bR^n\to\bR^m$ be a linear map, and let
$\cS_1$ be the set of all points $\mu_T(\uS)$
where $\uS$ is a self-similar integral $n$-system. Then,
$\cS_1$ is a dense subset of the spectrum $\image^*(\mu_T)$
and that spectrum is a connected subset of $\bR^m$.
\end{corollary}

Recall that an integral $n$-system is a
rigid $n$-system of mesh $1$.

\begin{proof}
Choose $C>0$ such that $\|T(\ux)\|_\infty\le C \|\ux\|_\infty$
for each $\ux\in\bR^n$.  For any proper
$n$-system $\uR$ and any $\epsilon>0$, the above theorem
provides a self-similar integral $n$-system $\uS$ with
$\dist(\cF(\uR),\cF(\uS))\le C^{-1}\epsilon$.  For this choice
of $\uS$, we have $\dist(T(\cF(\uR)),T(\cF(\uS)))\le \epsilon$
and so $\|\mu_T(\uR)-\mu_T(\uS)\|_\infty \le \epsilon$.
Thus $\cS_1$ is dense in $\image^*(\mu_T)$.
Since Corollary \ref{spectra:cor:prop:P^nu} shows that $\cS_1$
is contained in a single connected component of $\image^*(\mu_T)$,
the latter must be connected.
\end{proof}

%
%

\section{Compactness of the spectra}
\label{sec:composite}

In this section, we complete the proof of
Theorem \ref{intro:thm:muT} by showing that any spectrum
of exponents of approximation
(as in Definition \ref{intro:def:image_muT})
is compact. We first establish a general result, which
uses the following notation.

\begin{definition}
\label{composite:def}
Let $(F_i)_{i\ge 1}$ be a sequence of non-empty
subsets of $\bR^n$.
We denote by $\liminf_{i\to\infty} F_i$ the set of
all points in $\bR^n$ which can be written
as the limit of a sequence $(\ux_i)_{i\ge 1}$ in
$\prod_{i\ge 1}F_i$, and we denote by
$\limsup_{i\to\infty} F_i$ the set of all points
in $\bR^n$ which are an accumulation point of such a sequence
$(\ux_i)_{i\ge 1}$, that is, the limit of a subsequence
of $(\ux_i)_{i\ge 1}$.
\end{definition}

\begin{theorem}
\label{composite:thm}
Let $(\uR^{(i)})_{i\ge 1}$ be a sequence of proper rigid
$n$-systems of the same mesh $\delta$ for some $\delta>0$.
Suppose that there exists a convergent sequence
$(\ux^{(i)})_{i\ge 1}$ in $\prod_{i\ge 1}\cF(\uR^{(i)},\ue_1)$.
Then, there is a proper rigid $n$-system
$\uR$ of mesh $\delta$ such that
\begin{equation*}
\cF(\uR) = \limsup_{i\to\infty} \cF(\uR^{(i)}).
\end{equation*}
\end{theorem}

\begin{proof}
Choose a sequence of positive real numbers $(\epsilon_i)_{i\ge 0}$
such that $\lim_{i\to\infty}\epsilon_i=0$ and
$\epsilon_i\ge \|\ux^{(i)}-\ux^{(i+1)}\|_\infty$ for each
$i\ge 1$.  We first apply Proposition \ref{refined:prop:ajustement}
to construct, for each $i\ge 1$, positive multiples $u_i$ and $v_i$
of $\delta$ with $nu_i<v_i$ and a rigid $n$-system $\uP^{(i)}\colon
[u_i,v_i]\to\Delta_n$ of mesh $\delta$ which satisfy the following
four conditions:
\begin{flalign}
\label{composite:thm:eq1}
 \quad
 &(\uP^{(i)})'(u_i^+)=\ue_1, &\\
\label{composite:thm:eq2}
 &(1+9n\epsilon_{i-1})^{-1}(\ux^{(i)}+3\epsilon_{i-1}\ue)
  \le u_i^{-1}\uP^{(i)}(u_i)
  \le \ux^{(i)}+12\epsilon_{i-1}\ue,\\
\label{composite:thm:eq3}
 &\ux^{(i)}-2\epsilon_i\ue
  \le v_i^{-1}\uP^{(i)}(v_i)
  \le \ux^{(i)}+2\epsilon_i\ue,\\
\label{composite:thm:eq4}
 &\dist(\cF(\uR^{(i)}), E^{(i)}) \le 12(n+1)\epsilon_{i-1}
 \quad \text{where} \quad
 E^{(i)}=\{q^{-1}\uP^{(i)}(q)\,;\,q\in [u_i,v_i]\},
\end{flalign}
with $\ue$ given by \eqref{refined:eq:ue}.
More precisely, we start by making appropriate choices of $u_1,u_2,\dots$
Then, for each $i\ge 1$, we select $v_i$ large enough
compared to $u_{i+1}$ (with $v_i>nu_i$) so that the integers
defined recursively by $m_1=1$ and
$m_{i+1}=\left\lceil (1+9n\epsilon_i)m_iv_iu_{i+1}^{-1}\right\rceil$
for $i\ge 1$, satisfy
\begin{equation}
 \label{composite:thm:eq5}
 (1+9n\epsilon_i)m_iv_i \le m_{i+1}u_{i+1}\le (1+10n\epsilon_i)m_iv_i
 \quad (i\ge 1).
\end{equation}
Since $\ux^{(i+1)}-\epsilon_i\ue\le \ux^{(i)}\le \ux^{(i+1)}+\epsilon_i\ue$,
we deduce from \eqref{composite:thm:eq3}, \eqref{composite:thm:eq4}
and \eqref{composite:thm:eq5} that
\begin{align*}
 &m_iv_i(\ux^{(i+1)}-3\epsilon_i\ue)
  \le m_i\uP^{(i)}(v_i)
  \le m_iv_i(\ux^{(i+1)}+3\epsilon_i\ue),\\
 &m_iv_i(\ux^{(i+1)}+3\epsilon_i\ue)
  \le m_{i+1}\uP^{(i+1)}(u_{i+1})
  \le (1+10n\epsilon_i)m_iv_i(\ux^{(i+1)}+12\epsilon_i\ue),
\end{align*}
and, therefore, using $\ux^{(i+1)}\le \ue$, we conclude that
\begin{equation}
 \label{composite:thm:eq6}
 m_i\uP^{(i)}(v_i)
   \le m_{i+1}\uP^{(i+1)}(u_{i+1})
   \le m_i\uP^{(i)}(v_i)+\epsilon'_im_iv_i\ue,
\end{equation}
where $\epsilon'_i=\epsilon_i(10n+15+150n\epsilon_i)$.  For
each $i\ge 1$, define
\[
 \tu_i=m_iu_i
 \et
 \tv_i=m_iv_i.
\]
By \eqref{composite:thm:eq6} and the fact that $\tv_i>n\tu_i$,
Corollary \ref{deformation:cor:prop:extension}
applies to the rescaled rigid $n$-system
$\tuP^{(i)}\colon [\tu_i,\tv_i]\to\Delta_n$
of mesh $\delta$ given by
\[
 \tuP^{(i)}(q)=m_i\uP^{(i)}(m_i^{-1}q)
 \quad (q\in[\tu_i,\tv_i]),
\]
with the choice of $\uc=m_{i+1}\uP^{(i+1)}(u_{i+1})$.
It provides a rigid $n$-system $\tuR^{(i)}$ of mesh $\delta$
on $[\tu_i,\tu_{i+1}]$ such that
\begin{flalign}
 \label{composite:thm:eq7}
 \quad
 &\tuR^{(i)}(\tu_i)=m_i\uP^{(i)}(u_i)
  \et
  \big(\tuR^{(i)}\big)'(\tu_i^+)
     =\big(\uP^{(i)}\big)'(u_i^+)=\ue_1, &\\
 \label{composite:thm:eq8}
 &\tuR^{(i)}(\tu_{i+1}) = m_{i+1}\uP^{(i+1)}(u_{i+1}),\\
 \label{composite:thm:eq9}
 &\dist(E^{(i)},\tE^{(i)})\le n(n+1)\epsilon'_i
  \quad \text{where} \quad
  \tE^{(i)}=\{q^{-1}\tuR^{(i)}(q)\,;\, q\in[\tu_i,\tu_{i+1}]\},
\end{flalign}
the last property using the fact that
$\{q^{-1}\tuP^{(i)}(q)\,;\, q\in[\tu_i,\tv_i]\}=E^{(i)}$
since $\tuP^{(i)}$ is obtained from $\uP^{(i)}$ by rescaling.
By \eqref{composite:thm:eq7} and \eqref{composite:thm:eq8},
the $n$-systems $\tuR^{(1)},\tuR^{(2)},\dots$ can be pasted
together to produce a rigid $n$-system $\uR$ of mesh $\delta$
on $[\tu_1,\infty)$ whose restriction to $[\tu_i,\tu_{i+1}]$
is $\tuR^{(i)}$ for each $i\ge 1$.  By \eqref{composite:thm:eq4}
and \eqref{composite:thm:eq9}, it satisfies
\[
 \cF(\uR)
   = \limsup_{i\to \infty} \tE^{(i)}
   = \limsup_{i\to \infty}   E^{(i)}
   = \limsup_{i\to \infty}  \cF\big(\uR^{(i)}\big).
\qedhere
\]
\end{proof}

We note a first immediate consequence.

\begin{corollary}
\label{composite:cor1}
Let $\uR^{(1)}$ and $\uR^{(2)}$ be proper rigid $n$-systems
of the same mesh $\delta$.  Suppose that $\cF(\uR^{(1)},\ue_1)$
and $\cF(\uR^{(2)},\ue_1)$ have at least one point $\ux$ in
common. Then there exists a proper rigid $n$-system
$\uR$ of mesh $\delta$ such that
$\cF(\uR)=\cF(\uR^{(1)})\cup\cF(\uR^{(2)})$.
\end{corollary}

\begin{proof}
It suffices to apply Theorem \ref{composite:thm}
to the sequence $(\uR^{(i)})_{i\ge 1}$ given by
$\uR^{(i)}=\uR^{(1)}$ if $i$ is odd and $\uR^{(i)}=\uR^{(2)}$
if $i$ is even, using the constant sequence
$\ux^{(i)}=\ux$ for each $i\ge 1$. The resulting
$n$-system $\uR$ has the required property.
\end{proof}

If we drop the main assumption in the theorem,
we obtain the following weaker result.

\begin{corollary}
\label{composite:cor2}
Let $(\uP^{(i)})_{i\ge 1}$ be a sequence of proper $n$-systems.
Then, there is a proper $n$-system $\uR$ such that \
$\liminf_{i\to\infty} \cF(\uP^{(i)}) \subseteq
\cF(\uR) \subseteq \limsup_{i\to\infty} \cF(\uP^{(i)})$.
\end{corollary}

\begin{proof}
For each $i\ge 1$, let $\uR^{(i)}$ be an integral
$n$-system whose difference with $\uP^{(i)}$ is bounded,
and let $\ux^{(i)} \in \cF(\uR^{(i)},\ue_1)$.
Since $(\ux^{(i)})_{i\ge 1}$ is a sequence in $[0,1]^n$,
it contains a converging subsequence
$(\ux^{(i_j)})_{j\ge 1}$.  Then,
Theorem \ref{composite:thm} provides a proper integral
$n$-system $\uR$ such that
$\cF(\uR) = \limsup_{j\to\infty} \cF(\uR^{(i_j)})$.
Since $\cF(\uR^{(i)})=\cF(\uP^{(i)})$ for each $i\ge 1$,
this system has the required property.
\end{proof}

A similar but slightly more elaborate argument yields
the compactness of the spectra.

\begin{corollary}
\label{composite:cor3}
Let $T\colon\bR^n\to\bR^m$ be a linear map.
Then, the associated spectrum $\image^*(\mu_T)$ is a compact
subset of $\bR^m$.
\end{corollary}

\begin{proof}
Let $\uy$ be a point of $\bR^m$ in the topological closure
of $\image^*(\mu_T)$.  There exists
a sequence of proper integral $n$-systems $(\uR^{(i)})_{i\ge 1}$
such that $\inf T(\cF(\uR^{(i)}))$ converges to $\uy$ as
$i\to\infty$.  Choose $\ux^{(i)}\in\cF(\uR^{(i)},\ue_1)$
for each $i\ge 1$.  By going to a subsequence if necessary,
we may assume that $\ux^{(i)}$ converges to a point $\ux$
as $i\to\infty$.  Then, Theorem \ref{composite:thm}
provides a proper integral $n$-system $\uR$ such that
$\cF(\uR) = \limsup_{i\to\infty} \cF(\uR^{(i)})$.  For each
integer $j\ge 1$, define
\[
 F_j = \bigcup_{i\ge j} \cF\big(\uR^{(i)}\big).
\]
Since $\cF(\uR^{(i)})\subseteq [0,1]^n$ for each $i$,
the distance $\dist(F_j,\cF(\uR))$ tends to $0$ as $j\to\infty$.
Moreover, since $T$ is a linear map, there exists a
constant $C>0$ such that $\|T(\ux)\|_\infty\le C \|\ux\|_\infty$
for each $\ux\in\bR^n$.  So, we find that
\[
 \|\inf T(F_j)-\inf T(\cF(\uR))\|_\infty
  \le \dist\big( T(F_j),\, T(\cF(\uR)) \big)
  \le C\, \dist\big( F_j,\, \cF(\uR) \big)
\]
also tends to $0$ as $j\to\infty$.  On the other hand, the point
\[
 \inf T(F_j) = \inf\{ \inf T(\cF(\uR^{(i)})) \,;\, i\ge j\}
\]
converges to $\uy$ as $j\to\infty$.  Thus $\uy=\inf T(\cF(\uR))
=\mu_T(\uR)$ belongs to $\image^*(\mu_T)$.  This proves that
$\image^*(\mu_T)$ is a closed subset of $\bR^m$.  It is also
bounded and thus compact since, for $\uy$ and $\uR$ as above, we have
$\cF(\uR)\subseteq [0,1]^n$ and so $\|\uy\|_\infty\le C$.
\end{proof}

To complete the proof of Theorem \ref{intro:thm:muT}, we
use the following observation.

\begin{lemma}
\label{composite:lemma}
Let $T\colon\bR^n\to\bR^m$ be a linear map, and
let $\iota\colon\bR^{n-1}\to\bR^n$ be the linear map given
by $\iota(x_2,\dots,x_n)=(0,x_2,\dots,x_n)$ for any
$x_2,\dots,x_n\in\bR$.  Then, we have
\begin{equation}
\label{composite:lemma:eq}
 \image(\mu_T)
   =\begin{cases}
     \image^*(\mu_T) \cup \image(\mu_{T\circ\iota}) &\text{if $n\ge 3$,}\\
     \image^*(\mu_T) \cup \{T(0,1)\} &\text{if $n=2$.}
    \end{cases}
\end{equation}
\end{lemma}

\begin{proof}
Let $\uP=(P_1,\dots,P_n)$ be an arbitrary $n$-system with
unbounded domain $I$.  If $\uP$ is proper, then $\mu_T(\uP)$
belongs to $\image^*(\mu_T)$.  Otherwise, its component $P_1$ is
bounded and thus eventually constant, so there exist $q_0\in I$ and
$a\ge 0$ such that $P_1(q)=(n-1)a$ for each $q\in [q_0,\infty)$.
If $n\ge 3$, then the map $\uR\colon [q_0,\infty) \to\Delta_{n-1}$
given by
\[
 \uR(q)=(P_2(q)+a,\dots,P_n(q)+a) \quad (q\ge q_0)
\]
is an $(n-1)$-system, and we have $\cF(\uP)=\{0\}\times\cF(\uR)$,
thus $\mu_T(\uP)=\mu_{T\circ\iota}(\uR)$ belongs to
$\image(\mu_{T\circ\iota})$.
If $n=2$, then $P_2(q)=q-a$ for each $q\ge q_0$, thus
$\cF(\uP)=\{(0,1)\}$, and so $\mu_T(\uP)=T(0,1)$.
This shows that $\image(\mu_T)$ is contained in the right-hand
side of \eqref{composite:lemma:eq}.  The reverse inclusion
follows from the fact that, when $n=2$, the map $\uP\colon
[0,\infty) \to\Delta_2$ given by $\uP(q)=(0,q)$ for each $q\ge 0$
is a $2$-system with $\mu_T(\uP)=T(0,1)$, while, if $n\ge 3$,
the composite $\iota\circ\uR$ is an $n$-system with
$\mu_T(\iota\circ\uR)=\mu_{T\circ\iota}(\uR)$ for any
$(n-1)$-system $\uR$ with unbounded domain.
\end{proof}

\begin{proof}[Proof of Theorem \ref{intro:thm:muT}]
Let $T\colon\bR^n\to\bR^m$ be a linear map.  The connectedness
and the compactness of the spectrum $\image^*(\mu_T)$ follow
respectively from Corollaries \ref{self:cor} and \ref{composite:cor3}.
The compactness of $\image(\mu_T)$ then follows, by induction on $n$,
from the above lemma.  Finally, $\image(\mu_T)$ is not connected
if we take for example $n=2$ and $T\colon\bR^2\to\bR$ given
by $T(x_1,x_2)=x_2$
because, for any proper $2$-system $\uP=(P_1,P_2)$, we have
$P_1(q)=q-P_2(q)\le P_2(q)$ for each $q$ in the domain of $\uP$, with
equality for arbitrarily large values of $q$, thus
$\mu_T(\uP)=1/2$ which, by the above lemma, gives
$\image(\mu_T)=\{1/2,1\}$, a non-connected set.
 \end{proof}

%
%

\section{Self-similar non-degenerate $3$-systems}
\label{sec:3sys}

In this section, we provide a complete description of the sets
$\cF(\uS)$ attached to self-similar non-degenerate $3$-systems.
We conclude with a proof of Theorem \ref{intro:thm:u,v,w}
which, as we saw in the introduction, implies that,
in dimension $n=3$, any spectrum of exponents (as in Definition
\ref{intro:def:image_muT}) is closed under the minimum.

Let $\uP=(P_1,P_2,P_3)\colon [q_0,\infty) \to\Delta_3$ be a proper
non-degenerate $3$-system.  For each $q\ge q_0$,
the point $q^{-1}\uP(q)$ belongs to the triangle
\[
 \bar\Delta^{(3)}
 = \{ (x_1,x_2,x_3)\in\bR^3\,;\,
     0\le x_1\le x_2\le x_3 \text{ and } x_1+x_2+x_3=1\}
\]
with vertices
\[
 \uf_1=\frac{1}{2}(\ue_2+\ue_3), \quad
 \uf_2=\frac{1}{3}(\ue_1+\ue_2+\ue_3)
  \et
 \uf_3=\ue_3.
\]
So, the map
$\varphi_\uP\colon [q_0,\infty) \to\bar\Delta^{(3)}$
given by $\varphi_\uP(q)=q^{-1}\uP(q)$ $(q\ge q_0)$
represents a continuous path in that closed set.

When $q$ is a switch number of $\uP$, the coordinates
of $\uP(q)$ form a strictly increasing sequence of
positive numbers and so $q^{-1}\uP(q)$ is a
point in the relative interior of $\bar\Delta^{(3)}$
denoted
\[
 \Delta^{(3)}
 = \{ (x_1,x_2,x_3)\in\bR^3\,;\,
     0< x_1< x_2< x_3 \text{ and } x_1+x_2+x_3=1\}.
\]
When $q$ is a division number of $\uP$ which is not
a switch number, the point $\uP(q)$ is of the form
$(a,a,b)$ or $(a,b,b)$ with $0<a<b$, and
so $q^{-1}\uP(q)$ belongs to one of the
open line segments
\begin{align*}
 L&=\{(x_1,x_2,x_3)\in\bar\Delta^{(3)}\,;\,
   0<x_1=x_2<x_3\}
  =  (\uf_2,\uf_3), \\
\text{or}\quad
L^*&=\{(x_1,x_2,x_3)\in\bar\Delta^{(3)}\,;\,
   0<x_1<x_2=x_3\}
   = (\uf_2,\uf_1),
\end{align*}
using $[\ux,\uy]$, $[\ux,\uy)$, $(\ux,\uy]$
and $(\ux,\uy)$ as shorthand to denote the various
line segments between points $\ux$ and $\uy$ in
$\bR^3$, with $\ux$ and $\uy$ included or not according
to the same convention as for subintervals of $\bR$.

We also note that, when $u<v$ are consecutive division
numbers of $\uP$, all components of $\uP$ are constant
on $[u,v]$ except one, say $P_j$, which is strictly
increasing. Then, we have
\[
 \varphi_\uP([u,v]) 
  = \left[ u^{-1}\uP(u),\, v^{-1}\uP(v) \right]
  \subset \left[ u^{-1}\uP(u),\, \ue_j \right).
\]
It follows from this that $\varphi_\uP$ maps any compact
subinterval of $[q_0,\infty)$ to a polygonal chain
in $\bar\Delta^{(3)}$.  In particular, if $\uP$ is
self-similar, with $\uP(\rho q)=\rho\uP(q)$ for each
$q\ge q_0$ and some fixed $\rho>1$, then
$\varphi_\uP(\rho q)=\varphi_\uP(q)$ for each
$q\ge q_0$ and so $\cF(\uP)=\varphi_\uP([q_0,\rho q_0])$
is a closed polygonal chain.  This partly explains the
next result.

\begin{proposition}
\label{3sys:prop:chains}
The sets $\cF(\uS)$ where $\uS$ runs through
the self-similar non-degenerate $3$-systems are
the closed chains in $\bar\Delta^{(3)}$ defined
as follows.
\end{proposition}

\begin{definition}
\label{3sys:def:chains}
A \emph{simple chain} in $\bar\Delta^{(3)}$
from $A\in L$ to $A_1\in L$
is a polygonal chain of the form
\[
 A\, A^*_1\, C^*_1\, A^*_2\, C^*_2\, \cdots\, A^*_g\, C^*_g\,
 C_h\, A_h\, \cdots\, C_2\, A_2\, C_1\, A_1
\]
for some integers $g,h\ge 1$, where
\begin{alignat*}{3}
 &A^*_1\in L^*\cap [A, \ue_2],
       &&C_1^*\in \Delta^{(3)} \cap [A^*_1,\ue_3], \\
 &A^*_i\in L^* \cap [C^*_{i-1}, \ue_2],
       &&C_i^*\in \Delta^{(3)} \cap [A^*_i,\ue_3] \quad
       &&\text{for $i=2,\dots,g$,}\\
 &C_h\in \Delta^{(3)} \cap [C_g^*, \ue_2], \qquad
       &&A_h\in L \cap [C_h,\ue_1],\\
 &C_i\in \Delta^{(3)} \cap [A_{i+1}, \ue_2],
       &&A_i\in L \cap [C_i,\ue_1]
       &&\text{for $i=h-1,\dots,1$.}
\end{alignat*}
A \emph{closed chain} in $\bar\Delta^{(3)}$
is a closed polygonal chain which is a succession
of simple chains from $A^{(1)}$ to $A^{(2)}$,
$A^{(2)}$ to $A^{(3)}$, \dots, $A^{(s)}$ to $A^{(1)}$,
for some points $A^{(1)},\dots,A^{(s)}$ of
$L$ with $s\ge 2$.
\end{definition}

In both kinds of chain, any vertex, except the first,
lies on the line segment joining the preceding vertex
to $\ue_1$, $\ue_2$ or $\ue_3$.  Figure \ref{3sys:fig:schain}
below illustrates the notion of a simple chain.

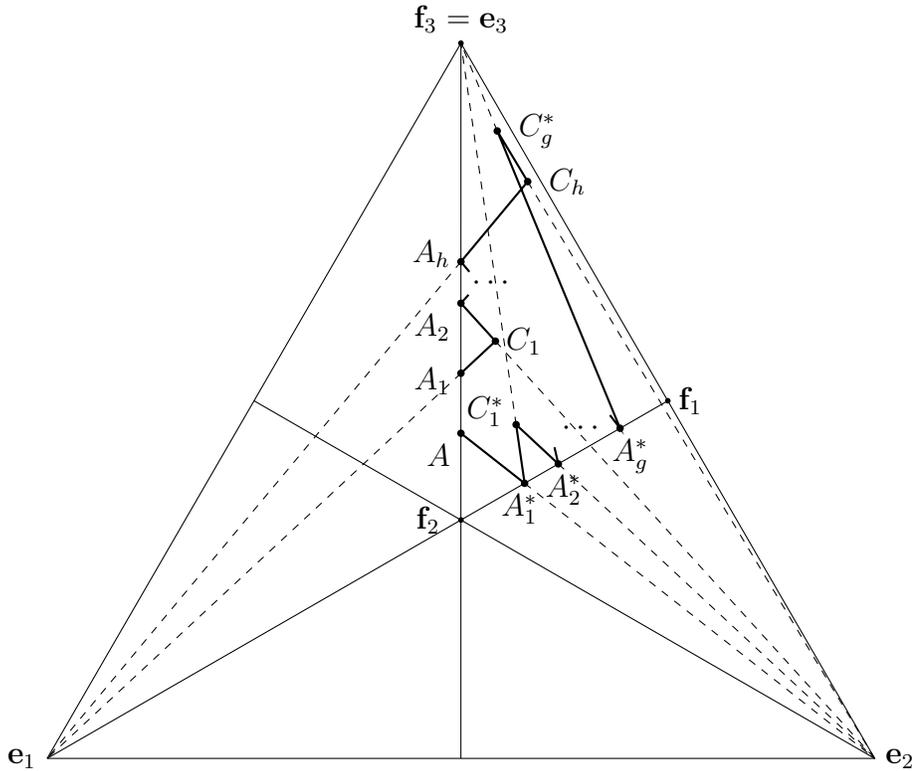
\begin{figure}[hb]
     \begin{tikzpicture}[scale=0.55]
       \coordinate (B) at (20,0);
       \coordinate (C) at (10,17.3);
       \coordinate (MBC) at ($0.5*(B)+0.5*(C)$);
       \coordinate (L23) at ($0.5*(B)+0.5*(C)$);
       \coordinate (M) at ($1/3*(B)+1/3*(C)$);
       \coordinate (L12) at ($0.35*(M)+0.65*(C)$);
       \draw[] (0,0)--(B)--(C)--(0,0);
       \draw[] (0,0)--(MBC);
       \node[circle, inner sep=0.75pt, fill] at (MBC) {};
       \node[right] at (MBC) {$\uf_1$};
       \draw[] (10,0)--(C);
       \draw[] ($0.5*(C)$)--(20,0);
       \node[left] at (0,0) {$\ue_1$};
       \node[right] at (B) {$\ue_2$};
       \node[circle, inner sep=0.75pt, fill] at (C) {};
       \node[above] at (C) {$\uf_3=\ue_3$};
       \node[circle, inner sep=0.75pt, fill] at (M) {};
       \node[left] at (M) {$\uf_2\ $};
       \coordinate (A2) at ($3/11*(B)+5/11*(C)$); 
       \coordinate (A3) at ($5/13*(B)+5/13*(C)$); 
       \coordinate (A4) at ($5/15*(B)+7/15*(C)$); 
       \coordinate (PA4) at ($(A4)-(1,0.3)$);
       \coordinate (A5) at ($7/17*(B)+7/17*(C)$); 
       \coordinate (E5) at ($7/17.7*(B)+7.7/17.7*(C)$); 
       \draw[thick] (A2)--(A3)--(A4)--(A5)--(E5);
       \draw[dashed] (A3)--(B);
       \draw[dashed] (A4)--(C);
       \draw[dashed] (A5)--(B);
       \node[circle, inner sep=1pt, fill] at (A2) {};
       \node[below left] at (A2) {$A$};
       \node[circle, inner sep=1pt, fill] at (A3) {};
       \node[below] at (A3) {$A^*_1\ $};
       \node[circle, inner sep=1pt, fill] at (A4) {};
       \node[above] at (PA4) {\ \ $C^*_1$};
       \node[circle, inner sep=1pt, fill] at (A5) {};
       \node[below] at (A5) {$\ A^*_2$};
       \coordinate (E9) at ($11/25*(B)+12/25*(C)$); 
       \coordinate (A9) at ($12/26*(B)+12/26*(C)$); 
       \coordinate (A10) at ($12/114*(B)+100/114*(C)$);
       \coordinate (A11) at ($22/124*(B)+100/124*(C)$);
       \coordinate (A12) at ($22/144*(B)+100/144*(C)$);
       \coordinate (PA12) at ($(A12)+(0,0.2)$);
       \coordinate (E12) at ($25/147*(B)+100/147*(C)$);
       \draw[thick] (E9)--(A9)--(A10)--(A11)--(A12)--(E12);
       \draw[dashed] (A10)--(C);
       \draw[dashed] (A11)--(B);
       \draw[dashed] (A12)--(0,0);
       \node[circle, inner sep=1pt, fill] at (A9) {};
       \node at (A9) [below]{$\ \ A^*_g$};
       \node at (A9) [left]{$\cdots\ $};
       \node[circle, inner sep=1pt, fill] at (A10) {};
       \node at (A10) [right]{$\ C^*_g$};
       \node[circle, inner sep=1pt, fill] at (A11) {};
       \node at (A11) [right]{$\ C_h$};
       \node[circle, inner sep=1pt, fill] at (A12) {};
       \node at (PA12) [left]{$A_h$};
       \coordinate (E13) at ($10/54*(B)+35/54*(C)$);
       \coordinate (A13) at ($10/55*(B)+35/55*(C)$);
       \coordinate (PA13) at ($1/2*(A12)+1/2*(A13)$);
       \coordinate (A14) at ($15/60*(B)+35/60*(C)$);
       \coordinate (A15) at ($15/65*(B)+35/65*(C)$);
       \coordinate (PA15) at ($(A15)+(0,0.4)$);
       \draw[thick] (E13)--(A13)--(A14)--(A15);
       \draw[dashed] (A14)--(B);
       \draw[dashed] (A15)--(0,0);
       \node at (PA13) [right]{$\cdots$};
       \node[circle, inner sep=1pt, fill] at (A13) {};
       \node at (A13) [below left]{$A_2$};
       \node[circle, inner sep=1pt, fill] at (A14) {};
       \node at (A14) [right]{$C_1$};
       \node[circle, inner sep=1pt, fill] at (A15) {};
       \node at (PA15) [below left]{$A_1$};
       \end{tikzpicture}
 \caption{A simple chain in $\bar\Delta^{(3)}$.}
 \label{3sys:fig:schain}
\end{figure}

\begin{proof}[Proof of Proposition \ref{3sys:prop:chains}]
Let $\big(\ua^{(i)}\big)_{i\ge 0}$ be a canvas in
$\bR^3$ whose associated $3$-system $\uS$
is self-similar and thus proper.  For each
$i\ge 0$, let $(k_i,\ell_{i+1})$ denote the pair of transition
indices defined by Condition (C2) from Section
\ref{sec:systems}.  By definition of a canvas, we have
$k_i<\ell_i$ for each $i\ge 1$, thus $k_i\neq 3$,
$\ell_i\neq 1$ and so
\[
 (k_i,\ell_{i+1})
   \in \{1,2\}\times\{2,3\}
   = \{(1,3),\, (2,3),\, (2,2),\, (1,2)\}
   \quad
   (i\ge 1).
\]
The condition $k_i\le \ell_{i+1}$ from (C2) is automatically
satisfied for these pairs.  The sequence
$\big((k_i,\ell_{i+1})\big)_{i\ge 1}$  can thus be viewed
as a walk in the following directed graph.
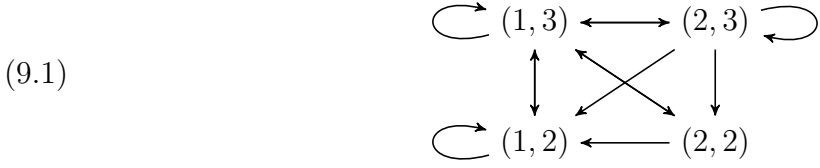
\begin{equation}
\label{3sys:eq:graph}
\begin{tikzpicture}[baseline=(current  bounding  box.center),
     ->,>=stealth',semithick,scale=0.8]
  \node (A) at (0,2) {$(1,3)$};
  \node (B) at (3,2) {$(2,3)$};
  \node (C) at (0,0) {$(1,2)$};
  \node (D) at (3,0) {$(2,2)$};
  \path (A) edge [loop left] (A) edge (B) edge (C) edge (D)
        (B) edge [loop right](B) edge (A) edge (C) edge (D)
        (C) edge [loop left] (C) edge (A)
        (D) edge (A) edge (C);
\end{tikzpicture}
\end{equation}
If there were only finitely many pairs equal to $(1,3)$, then
this sequence would eventually become constant and equal to $(2,3)$
or $(1,2)$, forcing the sequence $(a^{(i)}_1)_{i\ge 1}$
to be bounded, against the fact that $\uS$ is proper.
Thus, we have $(k_i,\ell_{i+1})=(1,3)$ for infinitely many
indices $i\ge 1$.

Consider two consecutive occurrences of the pair $(1,3)$,
say $(k_i,\ell_{i+1})=(k_j,\ell_{j+1})=(1,3)$ with
$1\le i<j$.  According to the above graph \eqref{3sys:eq:graph},
the intermediate pairs are
\[
 (k_i,\ell_{i+1})=(1,3),\,
 \underbrace{(2,3),\dots,(2,3)}_{\disp \text{$(g-1)$ times}},\,
 \big\{(2,2)\big\},\,
 \underbrace{(1,2),\dots,(1,2)}_{\disp \text{$(h-1)$ times}},\,
 (1,3)=(k_j,\ell_{j+1})
\]
for some integers $g,h\ge 1$, where the braces around
the pair $(2,2)$ indicate that this pair may or may not
appear in the sequence. The corresponding points
of the canvas are
\begin{equation}
\label{3sys:eq:liste:pts}
 \begin{aligned}
 \ua^{(i)} =
   &(c^*,a^*,b^*),\, (a^*,b^*,c^*_1),\, (a^*,c^*_1,c^*_2),\,
                      \dots\,, (a^*,c^*_{g-1},c^*_g=b),\\
   &(c_h=a^*,c_{h-1},b),\,\dots,\, (c_2,c_1,b),\,
         (c_1,a,b),\, (a,b,c)= \ua^{(j+1)},
 \end{aligned}
\end{equation}
for real numbers
\[
 0 < c^* < a^*=c_h < b^*=c^*_0 < c^*_1 < \cdots
< c^*_{g-1} \le c_{h-1}< \cdots < c_1 < a=c_0 < b=c^*_g < c
\]
with $c^*_{g-1}=c_{h-1}$ if there is no intermediate pair $(2,2)$.

Each pair of consecutive points in \eqref{3sys:eq:liste:pts} forms a
canvas and its associated $3$-system is the restriction of
$\uS$ to some compact interval $I$.  We describe below the corresponding
polygonal chain $F=\varphi_\uS(I)=\{q^{-1}\uS(q)\,;\,q\in I\}$
for each of them.
\begin{itemize}
 \item For the pair $\big((c^*,a^*,b^*),\, (a^*,b^*,c_1^*)\big)$ with
   transition indices $(1,3)$, there are two intermediate
   division points $(a^*,a^*,b^*)$ and $(a^*,b^*,b^*)$,
   so $F=CAA^*_1C^*_1$ where $C\in\Delta^{(3)}$,
   $A\in L\cap [C,\ue_1]$, $A^*_1\in L^* \cap [A,\ue_2]$ and
   $C^*_1\in \Delta^{(3)}\cap [A^*_1,\ue_3]$.
 \item For each of the pairs $\big((a^*,c^*_{i-2},c^*_{i-1}),\,
   (a^*,c^*_{i-1},c^*_i)\big)$
   $(2\le i\le g)$ with transition indices $(2,3)$, there is only
   one intermediate division point $(a^*,c^*_{i-1},c^*_{i-1})$,
   so $F = C^*_{i-1} A^*_i C^*_i$ where
   $A^*_i \in L^* \cap [C^*_{i-1},\ue_2]$ and
   $C^*_i\in\Delta^{(3)}\cap [A^*_i,\ue_3]$.
 \item For the pair $\big((a^*,c^*_{g-1},b),\, (a^*,c_{h-1},b)\big)$
   with transition indices $(2,2)$, there is
   no intermediate division point, so $F=C^*_gC_h$
   where $C_h \in \Delta^{(3)} \cap [C^*_g,\ue_2]$.
 \item For each of the pairs $\big((c_{i+1},c_i,b),\,
   (c_i,c_{i-1},b)\big)$
   $(h-1\ge i\ge 1)$ with transition indices $(1,2)$, there is only
   one intermediate division point $(c_i,c_i,b)$, so
   $F=C_{i+1}A_{i+1}C_i$ where $A_{i+1} \in L \cap [C_{i+1},\ue_1]$
   and $C_i \in \Delta^{(3)} \cap [A_{i+1},\ue_2]$.
 \item Finally, for the pair $\big((c_1,a,b),\,
   (a,b,c)\big)$ with transition indices $(1,3)$,
   the situation is the same as for the first pair, so $F=C_1A_1A'C'$
   where $A_1\in L\cap [C_1,\ue_1]$, $A'\in L^* \cap [A_1,\ue_2]$ and
   $C'\in\Delta^{(3)} \cap [A_1,\ue_3]$.
\end{itemize}
Thus $\varphi_\uS([2a^*+b^*,2a+b]) = A A^*_1 C^*_1\cdots C_1A_1$
is a simple chain in $\bar\Delta^{(3)}$, from $A$ to $A_1$,
as in Definition \ref{3sys:def:chains}.
Finally, since $\uS$ is self-similar, finitely many of these
chains suffice to cover the image of $\varphi_\uS$ and thus
the set $\cF(\uS)$ itself is a closed chain.

Conversely, suppose that $A_1A_2A_3\cdots A_{m+1}$
with $A_{m+1}=A_1$ is a
closed chain in $\bar\Delta^{(3)}$.  Then, for each
$i=1,\dots,m$, the point $A_{i+1}$ lies on the line segment
joining $A_i$ to $\ue_j$ for some $j$.  So, we can write
$A_{i+1}=\lambda_iA_i+(1-\lambda_i)\ue_j$ for some $\lambda_i
\in (0,1)$.  Then $A_{i+1}$ and $\lambda_i A_i$ have the same
$k$-th coordinate for each index $k$ distinct from $j$,
and the sum of the coordinates of $A_{i+1}$ exceeds that of
$\lambda_i A_i$.  Since $A_{m+1}=A_1$, the sequence
$(\tA_i)_{i\ge 1}$ defined recursively by
\[
 \tA_1=A_1,
 \quad
 \tA_{i+1}=(\lambda_1\dots\lambda_i)^{-1}A_{i+1}
 \  (1\le i\le m),
 \quad
 \tA_{i+m}=(\lambda_1\dots\lambda_m)^{-1}\tA_i
 \  (i\ge 1)
\]
has the property that consecutive points $\tA_i,\tA_{i+1}$
share two equal coordinates with the remaining coordinate
larger in $\tA_{i+1}$ than in $\tA_i$.  One checks that
the subsequence consisting of $\tA_1$ and the points $\tA_i$
with three different coordinates (those coming from points of
the open triangle $\Delta^{(3)}$) forms a canvas and that
$(\tA_i)_{i\ge 1}$ is the sequence of division points of its associated
non-degenerate $3$-system $\uS$. By construction, this
$3$-system is self-similar and $\cF(\uS)$ coincides with
the given closed chain.
\end{proof}

\begin{corollary}
\label{3sys:cor1}
Let $\uS^{(1)}$ and $\uS^{(2)}$ be a pair of
non-degenerate self-similar $3$-systems.  Define
$F=\cF(\uS^{(1)})\cup\cF(\uS^{(2)})$ and let $K$
denote the convex hull of $F$.  Then, there exists
a non-degenerate self-similar $3$-system $\uS$ with
$F\subseteq \cF(\uS)\subseteq K$.
\end{corollary}

In particular, for such a $3$-system $\uS$, we obtain
that $\cK(\uS)=K$ is also the convex hull of
$\cK(\uS^{(1)})\cup\cK(\uS^{(2)})$.

\medskip
\noindent
\begin{minipage}{.75\textwidth}
\begin{proof}
By the above proposition, each set $\cF(\uS^{(j)})$ is
a closed chain $A_jA^*_jC^*_j\cdots C_jA_j$ starting
with $A_j\in L$ and $A^*_j\in L^*\cap [A_j,\ue_2]$.
Without loss of generality, we may assume
that $A_1\in (\uf_2,A_2]$.  Then, we have
$A^*_1\in (\uf_2,A^*_2]$ as illustrated
in the figure on the right.
Let $C$ and $C^*$ denote
the points of $[A_2,A^*_2]$ which belong respectively
to the lines $\overleftrightarrow{\ue_1 A_1}$ and
$\overleftrightarrow{\ue_3 A^*_1}$.  Then
\[
 A_1A^*_1C^*_1\cdots C_1A_1\
 A^*_1C^*\
 A^*_2C^*_2\cdots C_2A_2\
 A^*_2C^*_2\cdots C_2A_2\
 CA_1
\]
is a closed chain containing $F$.  Since all its vertices
belong to $F$, it is also contained in $K$.  The
conclusion follows because, by the
proposition, this closed chain is equal to $\cF(\uS)$
for some non-degenerate self-similar $3$-system $\uS$.
\end{proof}
\end{minipage}
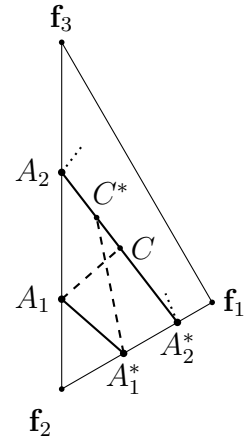
\begin{minipage}{.24\textwidth}
\centering
\def\x{2}%
\def\y{4.1}%
\def\u{2}%
\def\v{12}%
\def\e{1}%
\def\f{3}%
     \begin{tikzpicture}[scale=0.4] 
       \coordinate (B) at (20,0);
       \coordinate (C) at (10,17.3);
       \coordinate (MBC) at ($0.5*(B)+0.5*(C)$);
       \coordinate (M) at ($1/3*(B)+1/3*(C)$);
       \draw[] (M)--(MBC)--(C)--(M);
       \node[circle, inner sep=0.75pt, fill] at (C) {};
       \node[above] at (C) {$\uf_3$};
       \node[circle, inner sep=0.75pt, fill] at (MBC) {};
       \node[right] at (MBC) {$\uf_1$};
       \node[circle, inner sep=0.75pt, fill] at (M) {};
       \node[below] at ($(M)+(-0.7,-0.3)$) {$\uf_2$};
       \coordinate (A1) at ($\x/(2*\x+\y)*(B)+\y/(2*\x+\y)*(C)$); 
       \coordinate (A1e) at ($\y/(\x+2*\y)*(B)+\y/(\x+2*\y)*(C)$); 
       \coordinate (A2) at ($\u/(2*\u+\v)*(B)+\v/(2*\u+\v)*(C)$); 
       \coordinate (C2) at (${\u/(2*\u+\v-\e)}*(B)+{\v/(2*\u+\v-\e)}*(C)$); 
       \coordinate (A2e) at ($\v/(\u+2*\v)*(B)+\v/(\u+2*\v)*(C)$); 
       \coordinate (C2e) at (${\v/(\u+2*\v+\f)}*(B)+{(\v+\f)/(\u+2*\v+\f)}*(C)$); 
       \coordinate (Cc) at (${(\x/\y)/(\u/\v+\x/\y+1)}*(B)+{1/(\u/\v+\x/\y+1)}*(C)$);
       \coordinate (Ce) at (${(\y/\x)/(1+\y/\x+\v/\u)}*(B)+{(\v/\u)/(1+\y/\x+\v/\u)}*(C)$);
       \draw[thick] (A1)--(A1e);
       \draw[thick] (A2)--(A2e);
       \node[circle, inner sep=1pt, fill] at (A1) {};
       \node[left] at (A1) {$A_1$};
       \node[circle, inner sep=1pt, fill] at (A1e) {};
       \node[below] at (A1e) {$A^*_1$};
       \node[circle, inner sep=1pt, fill] at (A2) {};
       \node[left] at (A2) {$A_2$};
       \node[circle, inner sep=1pt, fill] at (A2e) {};
       \node[below] at (A2e) {$A^*_2$};
       \draw[thick, dashed] (A1)--(Cc);
       \node[circle, inner sep=0.75pt, fill] at (Cc) {};
       \node[right] at (Cc) {$C$};
       \draw[thick, dashed] (A1e)--(Ce);
       \node[circle, inner sep=0.75pt, fill] at (Ce) {};
       \node[right] at ($(Ce)+(-0.5,0.8)$) {$C^*$};
       \draw[thick, dotted] (A2)--(C2);
       \draw[thick, dotted] (A2e)--(C2e);
       \end{tikzpicture}
\end{minipage}

\medskip
In view of the considerations of Section \ref{sec:spectra},
the next result proves Theorem \ref{intro:thm:u,v,w}.

\begin{corollary}
\label{3sys:cor2}
Let $\uP^{(1)}$ and $\uP^{(2)}$ be proper
$3$-systems.  There exists a proper
$3$-system $\uP$ such that
$\cK(\uP)$ is the convex hull of
$\cK(\uP^{(1)})\cup\cK(\uP^{(2)})$.
\end{corollary}

\begin{proof}
Let $K$ denote the convex hull of the set
$F:=\cF(\uP^{(1)})\cup\cF(\uP^{(2)})$.
Using Theorem \ref{self:thm}, we choose, for each $j=1,2$,
a sequence $(\uS^{(i,j)})_{i\ge 1}$ of non-degenerate
self-similar $3$-systems such that $\dist(\cF(\uP^{(j)}),
\cF(\uS^{(i,j)}))$ tends to $0$ as $i\to\infty$.
Then, using the preceding corollary, we select for each $i\ge 1$,
a non-degenerate self-similar $3$-system $\uS^{(i)}$ such that
\[
 F_i\subseteq \cF(\uS^{(i)})\subseteq K_i,
\]
where $K_i$ is the convex hull of
$F_i:=\cF(\uS^{(i,1)})\cup\cF(\uS^{(i,2)})$.
By Corollary \ref{composite:cor2}, there is a proper
$n$-system $\uP$ such that
\[
 \liminf_{i\to\infty} \cF(\uS^{(i)}) \subseteq
\cF(\uP) \subseteq \limsup_{i\to\infty} \cF(\uS^{(i)}).
\]
Since $\dist(K_i,K) \le \dist(F_i,F)$
tends to $0$ as $i\to\infty$, we also have
\[
 F = \liminf_{i\to\infty} F_i
   \subseteq \liminf_{i\to\infty} \cF(\uS^{(i)})
 \et
 \limsup_{i\to\infty} \cF(\uS^{(i)})
   \subseteq \limsup_{i\to\infty} K_i = K.
\]
This implies that $F\subseteq \cF(\uP)\subseteq K$ and so
$\cK(\uP)=K$ is the convex hull of $F$ or, equivalently,
of $\cK(\uP^{(1)})\cup\cK(\uP^{(2)})$.
\end{proof}

%
%

\section{Semi-algebraicity of the spectra in dimension 3}
\label{sec:semi}

In this section, we characterize the sets $\cK(\uS)$ attached to
non-degenerate self-similar $3$-systems $\uS$ in terms of the
strict elementary paths defined below, and we use this to prove the
semi-algebraicity of the spectra in dimension $3$.

\begin{definition}
\label{semi:def:elementary path}
An \emph{elementary path} in $\bar\Delta^{(3)}$
is a polygonal chain $AA^*B^*C^*C B A$
with
\begin{alignat*}{3}
 &A\in \bar{L},
  &&A^*\in \bar{L}^*\cap [A,\ue_2],
  &&B^*\in [A^*, \uf_1],\\   
 &C^*\in \bar{\Delta}^{(3)} \cap [B^*,\ue_3], \quad
  &&C \in \bar{\Delta}^{(3)} \cap [C^*,\ue_2], \quad
  &&B\in [A,\ue_3] \cap [C,\ue_1],
\end{alignat*}
where $\bar{L}=[\uf_2,\uf_3]$ and $\bar{L}^*=[\uf_2,\uf_1]$
denote the closures of $L$ and $L^*$, respectively.
The \emph{base} of such a path is the line
segment $[A,A^*]$. We say that the
elementary path is \emph{strict} when
\[
 A,B\in L, \quad A^*,B^*\in L^* \et C,C^*\in\Delta^{(3)}.
\]
See Figure \ref{semi:fig:elem} below for an illustration.
\end{definition}

\begin{figure}[ht]
     \begin{tikzpicture}[scale=0.4] 
       \coordinate (B) at (20,0);
       \coordinate (C) at (10,17.3);
       \coordinate (MBC) at ($0.5*(B)+0.5*(C)$);
       \coordinate (M) at ($1/3*(B)+1/3*(C)$);
       \draw[] (0,0)--(B)--(C)--(0,0);
       \draw[] (0,0)--(MBC);
       \draw[] (10,0)--(C);
       \node[left] at (0,0) {$\ue_1$};
       \node[right] at (B) {$\ue_2$};
       \node[circle, inner sep=0.75pt, fill] at (C) {};
       \node[above] at (C) {$\uf_3=\ue_3$};
       \node[circle, inner sep=0.75pt, fill] at (MBC) {};
       \node[right] at (MBC) {$\uf_1$};
       \node[circle, inner sep=0.75pt, fill] at (M) {};
       \node[below] at ($(M)+(-0.7,-0.3)$) {$\uf_2$};
       \coordinate (A2) at ($3/11*(B)+5/11*(C)$); 
       \coordinate (A3) at ($5/13*(B)+5/13*(C)$); 
       \coordinate (A9) at ($8/18*(B)+8/18*(C)$); 
       \coordinate (A10) at ($8/70*(B)+60/70*(C)$);
       \coordinate (A11) at ($22/84*(B)+60/84*(C)$);
       \coordinate (A12) at ($22/104*(B)+60/104*(C)$);
       \coordinate (PA12) at ($(A12)+(0,0.2)$);
       \draw[thick] (A2)--(A3)--(A9);
       \draw[dashed] (A3)--(B);
       \node[circle, inner sep=1pt, fill] at (A2) {};
       \node[below left] at (A2) {$A$};
       \node[circle, inner sep=1pt, fill] at (A3) {};
       \node[below] at (A3) {$A^*$};
       \draw[thick] (A9)--(A10)--(A11)--(A12)--(A2);
       \draw[dashed] (A10)--(C);
       \draw[dashed] (A11)--(B);
       \draw[dashed] (A12)--(0,0);
       \node[circle, inner sep=1pt, fill] at (A9) {};
       \node at (A9) [below]{$\ B^*$};
       \node[circle, inner sep=1pt, fill] at (A10) {};
       \node at (A10) [right]{$\ C^*$};
       \node[circle, inner sep=1pt, fill] at (A11) {};
       \node at (A11) [right]{$\ C$};
       \node[circle, inner sep=1pt, fill] at (A12) {};
       \node at (PA12) [left]{$B$};
       \end{tikzpicture}
\caption{A strict elementary path in $\bar\Delta^{(3)}$.}
\label{semi:fig:elem}
\end{figure}
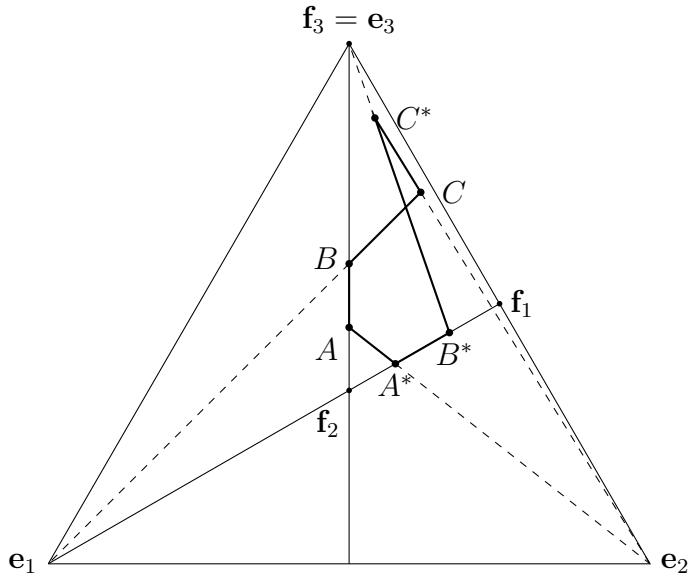

We start with the following observation.

\begin{lemma}
\label{semi:lemma}
Let $K$ be the convex hull of an elementary path
$\cE=AA^*B^*C^*CBA$ with base $AA^*$, and let
$V=\{A,B,C,A^*,B^*,C^*\}$ be the set of vertices of $\cE$.
Then there exists a proper $3$-system $\uP$ such
that $V\subseteq \cF(\uP)\subseteq K$ and thus
$\cK(\uP)=K$.  If $\cE$ is a strict elementary path,
we may choose $\uP$ to be self-similar and non-degenerate.
\end{lemma}

\begin{proof}
If $\cE$ is a strict elementary path,
we may insert points $C^*_1,A^*_2,\dots,C^*_m$
between $A^*$ and $B^*$, and points
$C_n,A_{n-1},\dots,C_1$ between $B$ and $A$,
with $C^*_1,\dots,C^*_m$ sufficiently close to $L^*$
and $C_n,\dots,C_1$ sufficiently close to $L$,
to obtain a simple closed chain
\[
 A\, A^*\, C^*_1\, A^*_2\, \cdots\, C^*_m\, B^*\, C^*\,
 C\, B\, C_n\, A_{n-1}\, \cdots\, C_1\, A
\]
whose convex hull is also $K$.  By Proposition
\ref{3sys:prop:chains}, this polygonal chain
coincides with the set $\cF(\uS)$ attached to
a non-degenerate self-similar $3$-system $\uS$,
and so $V\subseteq \cF(\uS)\subseteq K$.

In general, there exists a sequence of strict elementary
paths $\cE_i=A_iA_i^*B_i^*C_i^*C_iB_iA_i$ ($i\ge 1$)
whose vertices $A_i,\, A_i^*,\,\dots$ converge
respectively to $A,\, A^*,\,\dots$ as $i\to\infty$.
For each $i\ge 1$, we choose a non-degenerate
self-similar $3$-system $\uS^{(i)}$ with
$V_i\subseteq \cF(\uS^{(i)})\subseteq K_i$ where $V_i$
is the set of vertices of $\cE_i$, and $K_i$ is
its convex hull.  Then, by Corollary \ref{composite:cor2},
there exists a proper $3$-system $\uP$ such that
$\cF(\uP)$ is contained between
\[
 V=\liminf_{i\to\infty} V_i
   \subseteq \liminf_{i\to\infty}\cF(\uS^{(i)})
 \et
 \limsup_{i\to\infty}\cF(\uS^{(i)})
   \subseteq \limsup_{i\to\infty} K_i=K,
\]
as required.
\end{proof}

\begin{proposition}
\label{semi:prop}
A subset of $\bar\Delta^{(3)}$ is equal to
$\cK(\uS)$ for some non-degenerate
self-similar $3$-system $\uS$ if and only if it is
the convex hull of a finite non-empty
union of strict elementary paths in $\bar\Delta^{(3)}$.
\end{proposition}

\begin{proof}
If $\cE_1,\dots,\cE_s$ are strict elementary paths,
then the above lemma provides, for $i=1,\dots,s$,
a non-degenerate self-similar $3$-system $\uS^{(i)}$
such that $\cK(\uS^{(i)})$ is the convex hull $K_i$
of $\cE_i$, and Corollary \ref{3sys:cor1} implies
the existence of another non-degenerate self-similar
$3$-system $\uS$ such that $\cK(\uS)$ is the convex
hull of $K_1\cup\cdots\cup K_s$.  This shows that
the condition is sufficient.

Conversely, let $\uS$ be any non-degenerate self-similar
$3$-system.  By Proposition \ref{3sys:prop:chains}, the set
$\cF(\uS)$ is the union of finitely many simple chains
\begin{equation}
 \label{semi:eq:simple_chain}
 A A^*_1 C^*_1 \cdots A^*_g C^*_g C_h A_h \cdots C_1 A_1
\end{equation}
starting with $A\in L$ and $A^*_1\in L^*\cap[A,\ue_2]$.
Let $A_0$ (resp.\ $A^*_0$) denote the point of
$L\cap\cF(\uS)$ (resp.\ $L^*\cap\cF(\uS)$) which is
closest to $\uf_2$.  Then we have $A_0=A$
and $A^*_0=A^*_1$ for at least one of these simple
chains, and so $A^*_0\in L^*\cap[A_0,\ue_2]$.

Now, let $K$ be the convex hull of $A_0$, $A^*_0$
and of one of the simple chains
\eqref{semi:eq:simple_chain} composing $\cF(\uS)$.
We prove that $\cK(\uS)$ is the convex hull
of a finite union of strict elementary paths by
showing this for $K$.  First, we note that
$K$ contains the strict elementary path
$\cE=A_0A^*_0A^*_gC^*_gC_hA_hA_0$.  Moreover
the convex hull $K_0$ of $\cE$ contains $A_0$,
$A^*_0$ and all
vertices of the given simple chain except possibly
some points $C^*_i$ with $1\le i <g$ and some
points $C_i$ with $1\le i< h$.  Suppose that
$C^*_i\notin K_0$ for
some index $i$ with $1\le i <g$.  Then the line
$\overleftrightarrow{\ue_1 A_0}$ intersects the
segment $[C^*_i,A^*_{i+1}]$ in a point $\tC_i$,
and we obtain a strict elementary path
$A_0A^*_0A^*_iC^*_i\tC_iA_0$ contained in $\cK(\uP)$
(because $\tC_i\in\cF(\uP)$), and containing
$C^*_i$ (see Figure \ref{semi:fig:pointe*}).  Similarly, if
$C_i\notin K_0$ for some $i$ with $1\le i <h$,
then $\overleftrightarrow{\ue_3 A^*_0}$  
intersects $[A_{i+1},C_i]$
in a point $\tC^*_i$ yielding a strict elementary path
$A_0A^*_0\tC^*_iC_iA_iA_0$ both contained in $\cK(\uS)$
and containing $C_i$ (see Figure \ref{semi:fig:pointe}).
Thus $K$ is the convex hull of a finite union of strict
elementary paths.
\end{proof}

\noindent
\begin{minipage}{.5\textwidth}
  \centering
     \begin{tikzpicture}[scale=0.32] 
       \coordinate (B) at (20,0);
       \coordinate (C) at (10,17.3);
       \coordinate (MBC) at ($0.5*(B)+0.5*(C)$);
       \draw[] (0,0)--(B)--(C)--(0,0);
       \draw[] (0,0)--(MBC);
       \draw[] (10,0)--(C);
       \node[left] at (0,0) {$\ue_1$};
       \node[right] at (B) {$\ue_2$};
       \node[above] at (C) {$\ue_3$};
       \coordinate (A2) at ($3/11*(B)+5/11*(C)$); 
       \coordinate (A3) at ($5/13*(B)+5/13*(C)$); 
       \coordinate (A9) at ($6/14*(B)+6/14*(C)$); 
       \coordinate (A10) at ($6/28*(B)+20/28*(C)$);
       \coordinate (A11) at ($20/42*(B)+20/42*(C)$);
       \coordinate (A12) at ($12/34*(B)+20/34*(C)$);
       \draw[thick] (A2)--(A3)--(A9)--(A10)--(A12)--(A2);
       \draw[dashed] (0,0)--(A2);
       \draw[dashed] (A3)--(B);
       \draw[dashed] (A10)--(C);
       \draw[dashed] (A11)--(B);
       \draw[thick, dotted] (A12)--(A11);
       \node[circle, inner sep=1pt, fill] at (A2) {};
       \node[above left] at (A2) {$A_0$};
       \node[circle, inner sep=1pt, fill] at (A3) {};
       \node[below] at (A3) {$A^*_0\ $};
       \node[circle, inner sep=1pt, fill] at (A9) {};
       \node at (A9) [below]{$\ \ A^*_i$};
       \node[circle, inner sep=1pt, fill] at (A10) {};
       \node at (A10) [above right]{$\ C^*_i$};
       \node[circle, inner sep=1pt, fill] at (A11) {};
       \node[circle, inner sep=0.5pt, fill, color=white] at (A11) {};
       \node at (A11) [right]{$\ A^*_{i+1}$};
       \node[circle, inner sep=1pt, fill] at (A12) {};
       \node at (A12) [above right]{$\ \tC_i$};
       \end{tikzpicture}
  \captionof{figure}{}
  \label{semi:fig:pointe*}
\end{minipage}%
\begin{minipage}{.49\textwidth}
  \centering
     \begin{tikzpicture}[scale=0.32] 
       \coordinate (B) at (20,0);
       \coordinate (C) at (10,17.3);
       \coordinate (MBC) at ($0.5*(B)+0.5*(C)$);
       \draw[] (0,0)--(B)--(C)--(0,0);
       \draw[] (0,0)--(MBC);
       \draw[] (10,0)--(C);
       \node[left] at (0,0) {$\ue_1$};
       \node[right] at (B) {$\ue_2$};
       \node[above] at (C) {$\ue_3$};
       \coordinate (A2) at ($3/11*(B)+5/11*(C)$); 
       \coordinate (A3) at ($5/13*(B)+5/13*(C)$); 
       \coordinate (A4) at ($3/13*(B)+7/13*(C)$); 
       \coordinate (PA4) at ($(A4)+(0,0.5)$);
       \coordinate (A5) at ($3/11.2*(B)+7/11.2*(C)$);
       \coordinate (A6) at ($1.2/9.4*(B)+7/9.4*(C)$);
       \coordinate (PA6) at ($(A6)+(0.25,-0.7)$);
       \coordinate (A7) at ($2/10.2*(B)+7/10.2*(C)$);
       \coordinate (PA7) at ($(A7)+(-0.25,0.6)$);
       \draw[thick] (A3)--(A2)--(A4)--(A5)--(A7)--(A3);
       \draw[dashed] (0,0)--(A4);
       \draw[dashed] (A3)--(B);
       \draw[dashed] (A7)--(C);
       \draw[dashed] (A5)--(B);
       \draw[thick, dotted] (A7)--(A6);
       \node[circle, inner sep=1pt, fill] at (A2) {};
       \node[below left] at (A2) {$A_0$};
       \node[circle, inner sep=1pt, fill] at (A3) {};
       \node[below] at (A3) {$A^*_0\ $};
       \node[circle, inner sep=1pt, fill] at (A4) {};
       \node at (PA4) [left]{$A_i$};
       \node[circle, inner sep=1pt, fill] at (A5) {};
       \node at (A5) [right]{$C_i$};
       \node[circle, inner sep=1pt, fill] at (A6) {};
       \node[circle, inner sep=0.5pt, fill, color=white] at (A6) {};
       \node at (PA6) [left]{$A_{i+1}$};
       \node[circle, inner sep=1pt, fill] at (A7) {};
       \node at (PA7) [right]{$\tC^*_i$};
       \end{tikzpicture}
  \captionof{figure}{}
  \label{semi:fig:pointe}
\end{minipage}

\medskip
We now turn to the consequences in terms of spectra.

\begin{theorem}
\label{semi:thm}
Let $T=(T_1,\dots,T_m)$ be a linear map
from $\bR^3$ to $\bR^m$ for some integer $m\ge 1$,
and let $\uy=(y_1,\dots,y_m)\in \bR^m$.
Then, $\uy=\mu_T(\uS)$ for a non-degenerate
self-similar $3$-system $\uS$ if and only if there
exist strict elementary paths $\cE_1,\dots,\cE_m$
in $\bar\Delta^{(3)}$ such that
\begin{equation}
\label{semi:thm:eq1}
 y_j \le \inf(T_j(\cE_i)) \ (1\le i,j\le m)
   \et
 y_j = \inf(T_j(\cE_j)) \ (1\le j\le m).
\end{equation}
Moreover, $\uy=\mu_T(\uP)$ for a proper $3$-system
$\uP$ if and only if there
exist elementary paths $\cE_1,\dots,\cE_m$
in $\bar\Delta^{(3)}$ with the same property.
\end{theorem}

\begin{proof}
Suppose first that the conditions \eqref{semi:thm:eq1}
hold for some choice of elementary paths $\cE_1,\dots,\cE_m$.
Then, we have $\uy=\inf T(K)$ where $K$ is
the convex hull of $\cE_1\cup\cdots\cup\cE_m$.
If $\cE_1,\dots,\cE_m$ are strict, the preceding
proposition gives $K=\cK(\uS)$ for a non-degenerate
self-similar $3$-system $\uS$, and so $\uy=\mu_T(\uS)$.
In the general case, Lemma \ref{semi:lemma} shows that,
for $i=1,\dots,m$, the convex hull of $\cE_i$ is equal
to $\cK(\uP_i)$ for some proper $3$-system $\uP_i$.
By Corollary \ref{3sys:cor2}, this implies that $K=\cK(\uP)$
for a proper $3$-system $\uP$, and then $\uy=\mu_T(\uP)$.

Suppose now that $\uy=\mu_T(\uS)$ for a non-degenerate
self-similar $3$-system $\uS$.  By the preceding
proposition, $\cK(\uS)$ is the convex hull of a union
of strict elementary paths $\tcE_1,\dots,\tcE_s$.
For each $j=1,\dots,m$, we have
$y_j=\inf T_j(\tcE_1\cup\cdots\cup\tcE_s)$, thus
$y_j=\inf T_j(\cE_j)$ for at least one path $\cE_j$
among these.  Then $\cE_1,\dots,\cE_m$ fulfill
the conditions \eqref{semi:thm:eq1}.

Finally, suppose that $\uy=\mu_T(\uP)$ for a proper
$3$-system $\uP$.  By Corollary \ref{self:cor},
we have $\uy=\lim_{\ell\to\infty} \mu_T(\uS^{(\ell)})$
for a sequence of non-degenerate self-similar
$3$-systems $(\uS^{(\ell)})_{\ell\ge 1}$.  By the above,
this means that, for each $\ell\ge 1$, there exist
strict elementary paths $\cE_1^{(\ell)},\dots,\cE_m^{(\ell)}$
such that $y_j\le \inf T_j(\cE_i^{(\ell)})$ for each
$i,j=1,\dots,m$, with equality when $i=j$.  Write
\[
 \cE_i^{(\ell)}
  = A_i^{(\ell)} {A_i^*}^{(\ell)} {B_i^*}^{(\ell)}
    {C_i^*}^{(\ell)} C_i^{(\ell)} B_i^{(\ell)}
    A_i^{(\ell)}
 \quad
 (1\le i\le m,\ \ell\ge 1).
\]
Since the vertices of the paths $\cE_i^{(\ell)}$ belong
to the compact set $\bar\Delta^{(3)}$, there exist
integers $1\le \ell_1<\ell_2<\cdots$ such that
$A_i^{(\ell_j)}, {A_i^*}^{(\ell_j)}, \dots$ converge
respectively to points $A_i, A_i^*, \dots$ in
$\bar\Delta^{(3)}$ as $j\to\infty$, for each $i=1,\dots,m$.  Then
$\cE_i=A_iA^*_iB^*_iC^*_iC_iB_iA_i$ ($1\le i\le m$)
are elementary paths which, by continuity, fulfill
the conditions \eqref{semi:thm:eq1}
\end{proof}

As a consequence, we prove Theorem \ref{intro:thm:semi}
in the following stronger form.

\begin{corollary}
\label{3sys:cor:spectrum}
Let $T=(T_1,\dots,T_m)$ be as in Theorem \ref{semi:thm},
and let $\cS$ denote the set of all points $\mu_T(\uS)$
where $\uS$ runs through the non-degenerate self-similar
$3$-systems. Then both $\image^*(\mu_T)$ and $\cS$ are
semi-algebraic subsets of $\bR^m$.
\end{corollary}

\begin{proof}
By Theorem \ref{semi:thm}, a point
$\uy=(y_1,\dots,y_m)$ of $\bR^m$ belongs to
$\image^*(\mu_T)$ (resp.\ $\cS$) if and only if,
for each $i=1,\dots,m$, there exists an
elementary path (resp.\ a strict elementary path)
$\cE_i=A_iA^*_iB^*_iC^*_iC_iB_iA_i$
satisfying
\[
 \max\{ T_j(A_i), T_j(B_i), T_j(C_i), T_j(A^*_i),
        T_j(B^*_i), T_j(C^*_i)\}
 \le y_j
\]
for $j=1,\dots,m$, with equality for $j=i$. We view
this as a system of $m^2$ inequalities in the $18m$
coordinates of the $6m$ points $A_1,A_2,\dots,C^*_m$
and in the $m$ numbers $y_1,\dots,y_m$. For each
$i=1,\dots,m$, the condition that $\cE_i$ is an
elementary path (resp.\ a strict elementary path)
also translates into a system of
inequalities for the coordinates of its vertices
$A_i,B_i,\dots,C^*_i$.  All together these conditions
define a semi-algebraic subset of $\bR^{19m}$ of which
$\image^*(\mu_T)$ (resp.\ $\cS$) is the image under
the projection to the last $m$ coordinates.
By the Tarski-Seidenberg theorem \cite{Se1954},
the latter set is therefore a semi-algebraic
subset of $\bR^m$.
\end{proof}

%
%

\section{A special case}
\label{sec:six}

We now turn to the spectrum of the six
exponents $\phibot_1,\dots,\phitop_3$ in dimension $n=3$.
We first describe it in geometric terms, and
then as a semi-algebraic subset
of $\bR^6$. We conclude by proving its topological
property stated in Theorem \ref{intro:thm:S}.
For convenience, we denote by
\[
 \pi_1\colon
   \bar\Delta^{(3)}\longrightarrow\bar{L}
 \et
 \pi_3\colon
   \bar\Delta^{(3)}\setminus\{\ue_3\}
      \longrightarrow\bar{L}^*
\]
the projection operators with respective centers
$\ue_1$ and $\ue_3$, so that we have $\pi_1(A)\in
[A,\ue_1]$ and $A\in [\pi_3(A),\ue_3]$ for any
point $A$ in the appropriate domain.  We also denote
by $x_i$ the $i$-th coordinate function for $i=1,2,3$.
The following geometric observation is crucial
for what follows.

\begin{lemma}
\label{six:lemma:E}
Let $\cE=AA^*B^*C^*CBA$ be an elementary path
with base $AA^*$ in $\bar\Delta^{(3)}$.  Then, we have
\begin{align*}
 \inf \cE
  &= \big( x_1(C),\, \min\{x_2(C^*),x_2(B)\},\, x_3(A^*)\big),\\
 \sup \cE
  &= \big( x_1(A),\, \max\{x_2(B^*),x_2(C)\},\, x_3(C^*)\big).
\end{align*}
\end{lemma}

\begin{proof}
The level curves of the restriction of $x_1$ to the triangle
with vertices $\ue_1,\ue_2,\ue_3$ are line segments parallel
to $[\ue_2,\ue_3]$, with values increasing from $0$
on $[\ue_2,\ue_3]$ to $1$ on $\{\ue_1\}=[\ue_1,\ue_1]$.
In view of the slopes of the edges composing
$\cE$ (see Figure \ref{semi:fig:elem}),
we find that
\[
 x_1(A)\ge x_1(A^*)\ge x_1(B^*)\ge x_1(C^*)\ge x_1(C)
 \et
 x_1(A)\ge x_1(B)\ge x_1(C).
\]
So, the minimum of $x_1$ on $\cE$ is achieved
at the point $C$, and its maximum at the point $A$.
Similarly, the minimum of $x_3$ on $\cE$ is achieved
at $A^*$, and its maximum at $C^*$.
Finally, the minimum of $x_2$ on $\cE$ is achieved
at $B$ or at $C^*$, and its maximum at $B^*$ or at $C$.
\end{proof}

Theorem \ref{semi:thm} shows that one can realize
any point in the spectrum of $(\phibot_1,\dots,\phitop_3)$
using six elementary paths.  The next result together with
its dual shows more precisely that only two elementary paths
suffice.

\begin{proposition}
\label{six:prop:phibot1}
Let $(\pbot_1, \pbot_2, \pbot_3, \ptop_1,
\ptop_2, \ptop_3) \in \bR^6$.
There exists a proper $3$-system $\uP$ such that
\begin{equation}
\label{six:prop:phibot1:eq1}
 \begin{aligned}
 &\phitop_1(\uP)=\ptop_1, \
 \phibot_3(\uP)=\pbot_3, \
 \phibot_1(\uP)=\pbot_1, \
 \phitop_2(\uP)=\ptop_2,\\
 &\phibot_2(\uP)\ge \pbot_2, \
 \phitop_3(\uP)\le \ptop_3
 \end{aligned}
\end{equation}
if and only if there exists an elementary path
$\cE=AA^*B^*C^*CBA$ satisfying
\begin{equation}
\label{six:prop:phibot1:eq2}
 \begin{aligned}
 &x_1(A)=\ptop_1, \
  x_3(A^*)=\pbot_3, \
  x_2(B^*)=\ptop_2, \
  x_1(C)=\pbot_1,\\
 &x_2(C)\le \ptop_2, \
  x_2(B)\ge \pbot_2, \
  x_2(C^*)\ge \pbot_2, \
  x_3(C^*)\le \ptop_3,
 \end{aligned}
\end{equation}
and at least one of the two conditions $x_2(C)=\ptop_2$
or $B=A$, and, in the case where $B^*=\uf_1$,
the conditions $B=A$ and $C^*=C$.  Such an elementary
path, when it exists, is unique.
\end{proposition}

Note that, when $\pbot_2=0$ and $\ptop_3=1$, the
last two conditions of \eqref{six:prop:phibot1:eq1}
are automatically satisfied, as well as
the last three conditions of \eqref{six:prop:phibot1:eq2},
and the proposition yields a geometric
description of the spectrum of
$(\phitop_1,\phibot_3,\phibot_1,\phitop_2)$.
Moreover, the last assertion means that, when it exists, the
elementary path constructed by the proposition is a function
of $\ptop_1,\pbot_3,\pbot_1,\ptop_2$ alone.

\begin{proof}[Proof of Proposition \ref{six:prop:phibot1}]
Suppose first that there exists an elementary path
$\cE$ satisfying the conditions
\eqref{six:prop:phibot1:eq2}.  By Lemma \ref{semi:lemma},
there is a proper $3$-system $\uP$ for which $\cK(\uP)$
is the convex hull of $\cE$.  Then, $\cK(\uP)$ has
the same infimum and the same supremum as $\cE$, so
\[
 \phibot_i(\uP)=x_i(\inf \cK(\uP))=x_i(\inf \cE)
 \et
 \phitop_i(\uP)=x_i(\sup \cK(\uP))=x_i(\sup \cE)
 \quad (1\le i\le 3),
\]
and thus Lemma \ref{six:lemma:E} yields
\eqref{six:prop:phibot1:eq1}.

Conversely, suppose that there exists a proper $3$-system $\uP$
satisfying \eqref{six:prop:phibot1:eq1}. By Theorem \ref{semi:thm}
applied to the map $T\colon\bR^3\to\bR^6$ given by $T(\ux)=(\ux,-\ux)$
for each $\ux\in\bR^3$, there exist elementary paths
$\cE_i=A_iA^*_iB^*_iC^*_iC_iB_iA_i$ for $i=1,\dots,6$ such that
\[
 \mu_T(\uP)=\inf(T(\cE_1\cup\dots\cup\cE_6)).
\]
By Lemma \ref{six:lemma:E}, this means that
\[
 \begin{aligned}
 (\phibot_1(\uP),\phibot_2(\uP),\phibot_3(\uP))
   &=\min\left\{\big(x_1(C_i),\min\{x_2(B_i),x_2(C^*_i)\},x_3(A^*_i)\big)
        \,;\, 1\le i\le 6\right\},\\
 (\phitop_1(\uP),\phitop_2(\uP),\phitop_3(\uP))
   &=\max\left\{\big(x_1(A_i),\max\{x_2(B^*_i),x_2(C_i)\},x_3(C^*_i)\big)
        \,;\, 1\le i\le 6\right\}.
 \end{aligned}
\]
We choose $A\in\{A_1,\dots,A_6\}$ such that
\[
 x_1(A)=\max\{x_1(A_1),\dots,x_1(A_6)\}=\phitop_1(\uP)=\ptop_1.
\]
Then, the point $A^*$ on $[A,\ue_2]\cap\bar{L}^*$ satisfies
\[
 x_3(A^*)=\min\{x_3(A_1),\dots,x_3(A_6)\}=\phibot_3(\uP)=\pbot_3.
\]
We also choose $j,k\in \{1,\dots,6\}$ such that
\[
 x_1(C_j)=\phibot_1(\uP)=\pbot_1
 \et
 \max\{x_2(B^*_k),x_2(C_k)\}=\phitop_2(\uP)=\ptop_2.
\]

Since the plane of the equation $x_1=\pbot_1$ contains
$C_j \in \bar\Delta^{(3)}$, that plane cuts
$\bar\Delta^{(3)}$ in a line segment $[D,D^*]$ with $D\in \bar{L}$
and $D^*\in \bar{L^*}$, and $\cE_1,\dots,\cE_6$ are
contained in the triangle $\uf_2\/D^*\/D$.  In particular,
we have $A\in [\uf_2,D]$, $A^*\in [\uf_2,B^*]$, and so
$\cE_1,\dots,\cE_6$ are contained in the convex quadrilateral
$AA^*D^*D$.  Then, since the plane $x_2=\ptop_2$
contains the point $B^*_k$ or $C_k$, it must cut that
quadrilateral in a line segment $[B^*,E]$ with
$B^*\in[A^*,D^*]\subseteq \bar{L}^*$ (because
$x_2(A^*)\le \ptop_2$) and $E\in [C_j,D^*]\subseteq [D,D^*]$.
We conclude that $\cE_1,\dots,\cE_6$ are contained
in the convex polygon $A\/A^*\/B^*\/E\/D$.

Let $F$ denote the unique point of $[D,D^*]$ with $\pi_1(F)=A$,
and let $G=\pi_1(E)\in [\uf_2,D]\subseteq \bar{L}$.
If $G\in (A,D]$, we are in the situation of
Figure \ref{six:fig:phibot1:C=E}.  We have
$B^*\neq \uf_1$ because else $B^*=E$ and thus $G=\uf_2
\in [\uf_2,A]$, against the hypothesis.
Then, putting $B=G$ and $C=E$, we note that $[B^*,\ue_3]$
meets $[B,C]$, so there exists $C^*\in\bar\Delta^{(3)}$
for which $AA^*B^*C^*CBA$ is an elementary path.
Since $B^*\neq \uf_1$, the choice of $C^*$ is unique.
Otherwise, we are
in the situation of Figure \ref{six:fig:phibot1:B=A}
with $G\in [\uf_2,A]$.
Then, we define $B=A$, $C=F$, and we claim again,
that the line segment $[B^*,\ue_3]$ meets $[B,C]=[A,C]$.
This is clear if $G=A$ because then $C=F=E$.
Suppose now that $G\neq A$.  Then, we have $E\in[D^*,F)$
and so $x_2(C)=x_2(F)<x_2(E)=\ptop_2$.  Since
$C_k$ belongs to the triangle $ACD$ and
since $C$ is the point of that triangle with the
largest $x_2$ coordinate, this yields
$x_2(C_k)\le x_2(C)<\ptop_2$.  In view of the choice
of $k$, this implies that $x_2(B^*_k)=\ptop_2$,
and thus $B^*=B^*_k$. As $C^*_k$ also
belongs to $ACD$ while $B^*_k$ belongs to
$\bar{L}^*$, it follows that $[B^*_k,C^*_k]$ cuts
$[A,C]$.  A fortiori, the longer line segment
$[B^*_k,\ue_3]=[B^*,\ue_3]$ does the same.  So there is
a point $C^*$ in $\bar\Delta^{(3)}$ for which
$AA^*B^*C^*CBA$ is an elementary path.  If $B^*=\uf_1$,
we simply take $C^*=C$. Else, the choice of $C^*$ is unique.

\noindent
\begin{minipage}{.5\textwidth}
  \centering
     \begin{tikzpicture}[scale=0.65]
       \coordinate (E1) at (0,0);
       \coordinate (E2) at (20,0);
       \coordinate (E3) at (10,17.3);
       \coordinate (MBC) at ($0.5*(E2)+0.5*(E3)$);
       \coordinate (M) at ($1/3*(E2)+1/3*(E3)$);
       \draw[] (M)--(MBC)--(E3)--(M);
       \node[circle, inner sep=1pt, fill] at (E3) {};
       \node[above] at (E3) {$\ue_3$};
       \node[circle, inner sep=1pt, fill] at (M) {};
       \node[below left] at (M) {$\uf_2$};
       \coordinate (D) at ($0.85*(E3)+0.15*(M)$);
       \node[circle, inner sep=1pt, fill] at (D) {};
       \node[left] at (D) {$D$};
       \coordinate (De) at ($0.85*(MBC)+0.15*(M)$);
       \node[circle, inner sep=1pt, fill] at (De) {};
       \node[below right] at (De) {$D^*$};
       \coordinate (Dplus) at ($(D)+0.1*(E3)-0.1*(E2)$);
       \coordinate (Dmoins) at ($(De)-0.12*(E3)+0.12*(E2)$);
       \draw[dashed] (Dplus)--(Dmoins);
       \node at (Dmoins) [right]{$x_1=\pbot_1$};
       \coordinate (A) at ($10/32*(E1)+10/32*(E2)+12/32*(E3)$); 
       \coordinate (Ae) at ($10/34*(E1)+12/34*(E2)+12/34*(E3)$); 
       \node[circle, inner sep=1pt, fill] at (A) {};
       \coordinate (positionA) at ($(A)-(0,0.1)$);
       \node[left] at (positionA) {$A$};
       \node[circle, inner sep=1pt, fill] at (Ae) {};
       \coordinate (positionAe) at ($(Ae)+(0.2,0)$);
       \node[below] at (positionAe) {$A^*$};
       \coordinate (C) at ($0.05*(E1)+0.396*(E2)+0.554*(E3)$);
       \node[circle, inner sep=1pt, fill] at (C) {};
       \node[right] at (C) {$C=E$};
       \coordinate (Be) at ($0.396*(E2)+0.396*(E3)$);
       \node[circle, inner sep=1pt, fill] at (Be) {};
       \node at (Be) [below right]{$B^*$};
       \coordinate (phi2plus) at ($(Be)+0.36*(E3)$);
       \coordinate (phi2moins) at ($(Be)-0.02*(E3)$);
       \draw[dashed] (phi2moins)--(phi2plus);
       \node at (phi2plus) [right]{$x_2=\ptop_2$};
       \coordinate (B) at ($0.396/1.346*(E1)+0.396/1.346*(E2)+0.554/1.346*(E3)$);
       \node[circle, inner sep=1pt, fill] at (B) {};
       \node at (B) [left]{$B=G$};
       \coordinate (Ce) at ($0.072*(E1)+0.136*(E2)+0.792*(E3)$);
       \node[circle, inner sep=1pt, fill] at (Ce) {};
       \coordinate (positionCe) at ($(Ce)+(0.2,0.05)$);
       \node[right] at (positionCe) {$C^*$};
       \coordinate (Cj) at ($0.05*(E1)+0.30*(E2)+0.65*(E3)$);
       \node[circle, inner sep=1pt, fill] at (Cj) {};
       \coordinate (positionCj) at ($(Cj)+(-0.05,0.1)$);
       \node[right] at (positionCj) {$C_j$};
       \coordinate (Bej) at ($3/8*(E2)+3/8*(E3)$); 
       \node[circle, inner sep=1pt, fill] at (Bej) {};
       \coordinate (positionBej) at ($(Bej)+(0.4,0)$);
       \node[below] at (positionBej) {$B^*_j$};
       \coordinate (Bj) at ($0.3/1.25*(E2)+0.65/1.25*(E3)$); 
       \node[circle, inner sep=1pt, fill] at (Bj) {};
       \node[left] at (Bj) {$B_j$};
       \coordinate (F) at ($0.05*(E1)+{0.95*10/22}*(E2)+{0.95*12/22}*(E3)$);
       \node[circle, inner sep=1pt, fill] at (F) {};
       \coordinate (positionF) at ($(F)-(0.2,0)$);
       \node[right] at (F) {$F$};
       \coordinate (Cej) at ($0.0645*(E1)+0.0645*1.5*(E2)+0.0645*13*(E3)$);
       \node[circle, inner sep=1pt, fill] at (Cej) {};
       \coordinate (positionCej) at ($(Cej)+(0.05,0.1)$);
       \node[right] at (positionCej) {$C^*_j$};
       \draw[thick] (A)--(Ae)--(Be)--(Ce)--(C)--(B)--(A);
       \draw[thick, dashed] (Bej)--(Cej)--(Cj)--(Bj)--(A);
       \draw[dashed] (A)--(F);
       \end{tikzpicture}
  \captionof{figure}{}
  \label{six:fig:phibot1:C=E}
\end{minipage}
\begin{minipage}{.49\textwidth}
  \centering
     \begin{tikzpicture}[scale=0.65]
       \coordinate (E1) at (0,0);
       \coordinate (E2) at (20,0);
       \coordinate (E3) at (10,17.3);
       \coordinate (MBC) at ($0.5*(E2)+0.5*(E3)$);
       \coordinate (M) at ($1/3*(E2)+1/3*(E3)$);
       \draw[] (M)--(MBC)--(E3)--(M);
       \node[circle, inner sep=1pt, fill] at (M) {};
       \node[below left] at (M) {$\uf_2$};
       \node[circle, inner sep=1pt, fill] at (E3) {};
       \node[above] at (E3) {$\ue_3$};
       \coordinate (D) at ($0.85*(E3)+0.15*(M)$);
       \node[circle, inner sep=1pt, fill] at (D) {};
       \node[left] at (D) {$D$};
       \coordinate (De) at ($0.85*(MBC)+0.15*(M)$);
       \node[circle, inner sep=1pt, fill] at (De) {};
       \node[below right] at (De) {$D^*$};
       \coordinate (Dplus) at ($(D)+0.1*(E3)-0.1*(E2)$);
       \coordinate (Dmoins) at ($(De)-0.12*(E3)+0.12*(E2)$);
       \draw[dashed] (Dplus)--(Dmoins);
       \node at (Dmoins) [right]{$x_1=\pbot_1$};
       \coordinate (A) at ($10/33*(E1)+10/33*(E2)+13/33*(E3)$); 
       \coordinate (Ae) at ($10/36*(E1)+13/36*(E2)+13/36*(E3)$); 
       \node[circle, inner sep=1pt, fill] at (A) {};
       \coordinate (positionA) at ($(A)+(0,0.2)$);
       \node[left] at (positionA) {$B=A$};
       \node[circle, inner sep=1pt, fill] at (Ae) {};
       \node[below] at (Ae) {$A^*$};
       \coordinate (C) at ($0.05*(E1)+9.5/23*(E2)+13/23*0.95*(E3)$);
       \node[circle, inner sep=1pt, fill] at (C) {};
       \coordinate (positionC) at ($(C)+(-0.2,0.4)$);
       \node[right] at (positionC) {$C=F$};
       \coordinate (Be) at ($2/16*(E1)+7/16*(E2)+7/16*(E3)$); 
       \node[circle, inner sep=1pt, fill] at (Be) {};
       \node at (Be) [below right]{$B^*$};
       \coordinate (phi2plus) at ($(Be)+0.23*(E3)$);
       \coordinate (phi2moins) at ($(Be)-0.04*(E3)$);
       \draw[dashed] (phi2moins)--(phi2plus);
       \node at (phi2plus) [right]{$x_2=\ptop_2$};
       \coordinate (E) at ($0.05*(E1)+0.437*(E2)+0.513*(E3)$);
       \node[circle, inner sep=1pt, fill] at (E) {};
       \coordinate (positionE) at ($(E)+(0,0.1)$);
       \node[right] at (positionE) {$E$};
       \coordinate (G) at ($0.437/1.387*(E2)+0.513/1.387*(E3)$);
       \node[circle, inner sep=1pt, fill] at (G) {};
       \coordinate (positionG) at ($(G)+(0,-0.2)$);
       \node[left] at (positionG) {$G$};
       \coordinate (Ce) at ($0.066*(E1)+0.23*(E2)+0.705*(E3)$);
       \node[circle, inner sep=1pt, fill] at (Ce) {};
       \coordinate (positionCe) at ($(Ce)+(0.1,0.1)$);
       \node[right] at (positionCe) {$C^*$};
       \coordinate (Bej) at ($2/5*(E2)+2/5*(E3)$); 
       \node[circle, inner sep=1pt, fill] at (Bej) {};
       \node[below] at (Bej) {$B^*_j$};
       \coordinate (Cej) at ($2/16*(E2)+13/16*(E3)$); 
       \node[circle, inner sep=1pt, fill] at (Cej) {};
       \coordinate (positionCej) at ($(Cej)+(0.03,0.1)$);
       \node[right] at (positionCej) {$C^*_j$};
       \coordinate (Cj) at ($0.05*(E1)+0.3*(E2)+0.65*(E3)$);
       \node[circle, inner sep=1pt, fill] at (Cj) {};
       \coordinate (positionCj) at ($(Cj)+(-0.05,0.1)$);
       \node[right] at (positionCj) {$C_j$};
       \coordinate (Bj) at ($0.3/1.25*(E2)+0.65/1.25*(E3)$); 
       \node[circle, inner sep=1pt, fill] at (Bj) {};
       \node[left] at (Bj) {$B_j$};
       \draw[thick] (A)--(Ae)--(Be)--(Ce)--(C)--(A);
       \draw[thick, dashed] (Bej)--(Cej)--(Cj)--(Bj)--(A);
       \draw[dashed] (G)--(E);
       \end{tikzpicture}
  \captionof{figure}{}
  \label{six:fig:phibot1:B=A}
\end{minipage}%

In all cases, the triangle $B^*C^*\ue_2$ is contained in
the triangle $B_j^*C_j^*\ue_2$ because $B^*\in[B^*_j,D^*]$
and $C\in[C_j,D^*]$, as illustrated in Figures
\ref{six:fig:phibot1:C=E} and \ref{six:fig:phibot1:B=A}.
Since $C^*_j$ is the point of that triangle with
the largest $x_3$-coordinate and also the one with the
smallest $x_2$-coordinate, we deduce that
\[
 x_3(C^*)\le x_3(C^*_j)\le \phitop_3(\uP)
 \et
 x_2(C^*)\ge x_2(C^*_j)\ge \phibot_2(\uP).
\]
Finally, we have $B=\pi_1(C)\in \pi_1([C_j,D^*])=[B_j,\uf_2]$
and thus
\[
 x_2(B)\ge x_2(B_j)\ge \phibot_2(\uP).
\]
So, all conditions in \eqref{six:prop:phibot1:eq2} are fulfilled.
\end{proof}

\medskip
The next proposition is dual to the above in the sense
that it is obtained from it by permuting $A$ and $A^*$,
$B$ and $B^*$, $C$ and $C^*$, $x_1$ and $x_3$, $\phibot_1$
and $\phitop_3$, $\phibot_2$ and $\phitop_2$, $\phibot_3$
and $\phitop_1$, $\pbot_1$ and $\ptop_3$, $\pbot_2$ and
$\ptop_2$, $\pbot_3$ and $\ptop_1$, and by reversing
inequalities.   The proof is also dual in that sense,
although some slight additional modifications have
to be done.  This is left to the reader.  Note that, in
that sense Lemma \ref{six:lemma:E} is auto-dual as it
stays the same under these permutations, provided
that we permute as well the functions $\min$ and $\max$,
$\inf$ and $\sup$.

\begin{proposition}
\label{six:prop:phitop3}
Let $(\pbot_1, \pbot_2, \pbot_3, \ptop_1,
\ptop_2, \ptop_3) \in \bR^6$.
There exists a proper $3$-system $\uP$ such that
\begin{equation}
\label{six:prop:phitop3:eq1}
 \begin{aligned}
 &\phitop_1(\uP)=\ptop_1, \
 \phibot_3(\uP)=\pbot_3, \
 \phibot_2(\uP)=\pbot_2, \
 \phitop_3(\uP)=\ptop_3,\\
 &\phibot_1(\uP)\ge \pbot_1, \
 \phitop_2(\uP)\le \ptop_2
 \end{aligned}
\end{equation}
if and only if there exists an elementary path
$\cE=AA^*B^*C^*CBA$ satisfying
\begin{equation}
\label{six:prop:phitop3:eq2}
 \begin{aligned}
 &x_1(A)=\ptop_1, \
  x_3(A^*)=\pbot_3, \
  x_2(B)=\pbot_2, \
  x_3(C^*)=\ptop_3,\\
 &x_2(C^*)\ge \pbot_2, \
  x_2(B^*)\le \ptop_2, \
  x_2(C)\le \ptop_2, \
  x_1(C)\ge \pbot_1,
 \end{aligned}
\end{equation}
and at least one of the two conditions $x_2(C^*)=\pbot_2$
or $B^*=A^*$, and, in the case where $B=\uf_3$,
the conditions $B^*=A^*$ and $C=C^*$.  Such an elementary
path, when it exists, is unique.
\end{proposition}

Again, for the choice of $\pbot_1=0$ and $\ptop_2=1$, the
last two conditions of \eqref{six:prop:phitop3:eq1}
and the last three conditions of \eqref{six:prop:phitop3:eq2}
are automatically satisfied, yielding a geometric
description of the spectrum of
$(\phitop_1,\phibot_3,\phibot_2,\phitop_3)$.
We also deduce that, when it exists, the
elementary path constructed by the proposition is a function
of $\ptop_1,\pbot_3,\pbot_2,\ptop_3$ alone.

Putting the two propositions together, we obtain the following
geometric description of the spectrum of the six exponents.

\begin{corollary}
\label{six:cor:spectrum}
A point $\ualpha=(\pbot_1, \pbot_2, \pbot_3, \ptop_1,
\ptop_2, \ptop_3) \in \bR^6$ belongs to the spectrum $\cS$
of the family $(\phibot_1, \phibot_2, \phibot_3, \phitop_1,
\phitop_2, \phitop_3)$ if and only if, for the same point,
there exist elementary paths $\cE_1$ and $\cE_2$
satisfying respectively the conditions of Proposition
\ref{six:prop:phibot1} and \ref{six:prop:phitop3}.
\end{corollary}

For the proof, recall from the introduction that $\cS=\sigma(\image^*(\mu_T))$
where $T\colon\bR^3\to\bR^6$ and $\sigma\colon\bR^6\to\bR^6$
are given by $T(\ux)=(\ux,-\ux)$ and $\sigma(\ux,\uy)=(\ux,-\uy)$
for each $\ux, \uy \in \bR^3$.  Since $\image^*(\mu_T)$ is closed
under the minimum, we deduce that, for any points $(\ux_1,\uy_1),
(\ux_2,\uy_2)$ of $\cS$, viewed as a subset of $\bR^3\times\bR^3$,
we have $(\min(\ux_1,\ux_2),\max(\uy_1,\uy_2))\in\cS$.

\begin{proof}[Proof of Corollary \ref{six:cor:spectrum}]
Clearly the existence of such elementary paths is necessary.
Conversely, suppose the existence of such paths.
Let $\uP_1$ and $\uP_2$ be corresponding proper $3$-systems
satisfying \eqref{six:prop:phibot1:eq1} and
\eqref{six:prop:phitop3:eq1} respectively.  Since
\[
 \min(\phibot_i(\uP_1),\phibot_i(\uP_2))=\pbot_i
 \et
 \max(\phitop_i(\uP_1),\phitop_i(\uP_2))=\ptop_i
 \quad
 \text{for $i=1,2,3,$}
\]
the remark made before the proof implies that
$\ualpha\in\cS$.
\end{proof}

We now turn to an explicit description of the spectrum 
$\cS$ as a semi-algebraic set.  Although the statements
of Propositions \ref{six:prop:phibot1} and \ref{six:prop:phitop3}
suggest multiple systems of inequalities depending on the various
forms of the elementary paths, we manage to construct a single
set of inequalities.

\begin{theorem}
\label{six:thm:ineq}
Let $\ualpha=(\pbot_1, \pbot_2, \pbot_3, \ptop_1,
\ptop_2, \ptop_3) \in \bR^6$.  The point $\ualpha$ satisfies
the conditions of Proposition \ref{six:prop:phibot1} if
and only if the following hold:
\begin{align}
\label{six:thm:ineq:eq1}
 &\pbot_1\le \ptop_1\le 1/3\le \pbot_3\le \ptop_2\le (1-\pbot_1)/2 \le x_2(\uf_1),\\
\label{six:thm:ineq:eq2}
 &(1-2\ptop_1)(1-2\pbot_3)=\ptop_1\pbot_3,\\
\label{six:thm:ineq:eq3}
 &(\pbot_1+2\ptop_1-3\pbot_1\ptop_1)\ptop_2\le (1-\pbot_1)\ptop_1,\\
\label{six:thm:ineq:eq4}
 &\pbot_2\le \ptop_1,
  \quad
  (1-\pbot_1+\ptop_2)\pbot_2\le \ptop_2,\\
\label{six:thm:ineq:eq5}
 &\beta\pbot_2\le \pbot_1\ptop_2,
  \quad
 \gamma\pbot_2\le (1-\ptop_1)\pbot_1\ptop_2,\\
\label{six:thm:ineq:eq6}
 &\beta\ptop_3\ge (1-2\ptop_2)(1-\pbot_1-\ptop_2),
  \quad
 \gamma\ptop_3\ge (1-\pbot_1)(1-2\ptop_1)(1-2\ptop_2),\\
\label{six:thm:ineq:eq7}
 &(1-\ptop_1)\ptop_3\ge (1-\pbot_1)(1-2\ptop_1),
\end{align}
where
\begin{align*}
 \beta=\pbot_1\ptop_2+(1-2\ptop_2)(1-\ptop_2),\quad
 \gamma=\pbot_1(1-\ptop_1)(1-\ptop_2)+(1-\pbot_1)(1-2\ptop_1)(1-2\ptop_2).
\end{align*}

The same point $\ualpha$ fulfills the conditions
of Proposition \ref{six:prop:phitop3}
if and only if it satisfies $\pbot_2+\ptop_3\le 1$
and all the dual constraints obtained from \eqref{six:thm:ineq:eq1}
to \eqref{six:thm:ineq:eq7} by interchanging everywhere
the symbols $\pbot_1$ and $\ptop_3$, $\pbot_2$ and $\ptop_2$,
$\pbot_3$ and $\ptop_1$, the inequality
signs $\le$ and $\ge $, and the constants $x_2(\uf_1)=1/2$ and
$x_2(\uf_3)=0$.

Finally $\ualpha$ belongs to the spectrum $\cS$ of
the six exponents if and only if it satisfies the constraints
\eqref{six:thm:ineq:eq1} to \eqref{six:thm:ineq:eq7}, the dual
constraints indicated above,
and the additional condition $\pbot_2+\ptop_3\le 1$.
\end{theorem}

As a consequence, we deduce that \eqref{six:thm:ineq:eq1} to
\eqref{six:thm:ineq:eq3} are necessary and sufficient conditions
for a point $(\ptop_1,\pbot_3,\pbot_1,\ptop_2)$ to be in the
spectrum of $(\phitop_1,\phibot_3,\phibot_1,\phitop_2)$.  This
follows from the remark made after Proposition \ref{six:prop:phibot1}
upon noting that the other inequalities are satisfied when
$\pbot_2=0$ and $\ptop_3=1$. Similarly, the duals of these
constraints together with $\pbot_2+\ptop_3\le 1$
are necessary and sufficient conditions
for a point $(\ptop_1,\pbot_3,\pbot_2,\ptop_3)$ to be in the
spectrum of $(\phitop_1,\phibot_3,\phibot_2,\phitop_3)$.

In \cite{La2009}, Laurent showed that the spectrum of the exponents
$(\phibot_1,\phibot_3,\phitop_1,\phitop_3)$ consists of the points
$(\pbot_1,\pbot_3,\ptop_1,\ptop_3)$ in $\bR^4$ satisfying $0\le \pbot_1\le
\ptop_1\le 1/3$, Jarn\'{\i}k's condition \eqref{six:thm:ineq:eq2},
as well as \eqref{six:thm:ineq:eq7} and its dual.  Very recently,
in \cite{SS2016b}, Schmidt and Summerer have independently established
the inequality \eqref{six:thm:ineq:eq3} and its dual.  However, we must
warn the reader that our definition of the exponents $\phibot_i$ and
$\phitop_i$ is slightly different so that, for each $i=1,2,3$,
what they write as $\phibot_i$ (resp.\ $\phitop_i$) becomes $1-3\phitop_{4-i}$
 (resp.\ $1-3\phibot_{4-i}$) in our notation.

\begin{proof}[Proof of Theorem \ref{six:thm:ineq}]
We first show that the conditions \eqref{six:thm:ineq:eq1} to
\eqref{six:thm:ineq:eq7} are equivalent
to the existence of an elementary path $AA^*B^*C^*CBA$
as in Proposition \ref{six:prop:phibot1}.
To begin, we note that the first three conditions in
\eqref{six:prop:phibot1:eq2} determine uniquely $A$, $A^*$ and $B^*$:
\[
 A = (\ptop_1, \ptop_1, 1-2\ptop_1),\quad
 A^* = (1-2\pbot_3, \pbot_3, \pbot_3),\quad
 B^* = (1-2\ptop_2, \ptop_2, \ptop_2).
\]
Then \eqref{six:thm:ineq:eq1} and \eqref{six:thm:ineq:eq2}
simply translate the requirements that
\[
 A  \in [\uf_2,D]
    \subseteq \bar{L}=[\uf_2,\uf_3],
 \quad
 A^*\in [\uf_2,B^*]
    \subseteq [\uf_2,D^*]
    \subseteq \bar{L}^*=[\uf_2,\uf_1]
 \et
 A^*\in [A,\ue_2],
\]
where
\[
 D=(\pbot_1,\pbot_1,1-2\pbot_1)
 \et
 D^*=(\pbot_1,(1-\pbot_1)/2,(1-\pbot_1)/2).
\]
In particular, \eqref{six:thm:ineq:eq1} implies that
there exists a unique choice of points $E,F\in [D,D^*]$ satisfying
$x_2(E)=\ptop_2$ and $\pi_1(F)=A$.  They are given by
\[
 E=(\pbot_1,\ptop_2,1-\pbot_1-\ptop_2)
 \et
 F=\left( \pbot_1, \frac{(1-\pbot_1)\ptop_1}{1-\ptop_1},
       \frac{(1-\pbot_1)(1-2\ptop_1)}{1-\ptop_1} \right).
\]
Since $E\in\bar\Delta^{(3)}$, we may also form
\[
 G=\pi_1(E)
  =(1-\pbot_1+\ptop_2)^{-1}(\ptop_2, \ptop_2, 1-\pbot_1-\ptop_2)
  \in \bar{L}.
\]

Suppose first that $B^*\neq\uf_1$ or, equivalently, that
$\ptop_2< 1/2$.  The proof of Proposition \ref{six:prop:phibot1}
shows that, if there exists an elementary path
$\cE=AA^*B^*C^*CBA$ with the requested properties,
then the points $B$, $C$ and $C^*$ are given by
\[
\begin{cases}
 C=E,\ B=G,\ C^*=C^*_E
   &\text{if $G\in (A,D]$, as in
      Figure \ref{six:fig:phibot1:C=E},}\\
 C=F,\ B=A,\ C^*=C^*_F
   &\text{if $G\in [\uf_2,A]$, as in
      Figure \ref{six:fig:phibot1:B=A},}
\end{cases}
\]
where $C^*_E$ (resp.\ $C^*_F$) denotes the point of
intersection of $\overleftrightarrow{\ue_3B^*}$ with
the non-parallel line $\overleftrightarrow{\ue_2E}$
(resp.\ $\overleftrightarrow{\ue_2F}$).  We find that
\begin{align*}
 C^*_E
  &=\beta^{-1}\big(\, \pbot_1(1-2\ptop_2),\,
     \pbot_1\ptop_2,\, (1-\pbot_1-\ptop_2)(1-2\ptop_2)\,\big),\\
 C^*_F
  &=\gamma^{-1}\big(\, \pbot_1(1-\ptop_1)(1-2\ptop_2),\,
     \pbot_1\ptop_2(1-\ptop_1),\, (1-\pbot_1)(1-2\ptop_1)(1-2\ptop_2)\,\big),
\end{align*}
where $\beta,\gamma>0$ are as in the statement of the
proposition.  In order for these points to make
an elementary path, we simply need (besides
\eqref{six:thm:ineq:eq1} and \eqref{six:thm:ineq:eq2})
that $C$ belongs to the triangle $\ue_3B^*\uf_1$.  Since
$D^*$ already belongs to that triangle and since
$C\in\{E,F\}\subset [D^*,C]$, this is equivalent to asking that
both $E$ and $F$ belong to that triangle.  This is
automatic for $E$, because of the geometry.
For the point $F$, this
is equivalent to asking that $F=\ue_3$ or that
$\pi_3(F)\in [B^*,\uf_1]$.  Since
\[
 B^*\neq \uf_1
   \ \Rightarrow\ A^*\neq\uf_1
   \ \Rightarrow\ \pi_1(F)=A\neq\ue_3
   \ \Rightarrow\ F\neq\ue_3,
\]
this condition reduces to $x_2(B^*)\le x_2(\pi_3(F))$, which translates
into \eqref{six:thm:ineq:eq3}.  The conditions $x_1(C)=\pbot_1$
and $x_2(C)\le \ptop_2$ follow from the choice of $C$.  We also
have $B\in\{A,G\}$ and $A,G\in [\uf_2,B]$.  Thus the condition
$x_2(B)\ge \pbot_2$ is equivalent to both $x_2(A)\ge \pbot_2$ and
$x_2(G)\ge \pbot_2$ which translate into \eqref{six:thm:ineq:eq4}.
Similarly, we have $C^*\in\{C^*_E,C^*_F\}$ and $C^*_E,C^*_F\in
[B^*,C^*]$, because $E,F\in[C,D^*]$.  Thus, the condition
$x_2(C^*)\ge \pbot_2$ is equivalent to both
$x_2(C_E^*)\ge \pbot_2$ and $x_2(C_F^*)\ge \pbot_2$ which is
\eqref{six:thm:ineq:eq5}, while $x_3(C^*)\le \ptop_3$
is equivalent to both $x_3(C_E^*)\le \ptop_3$ and
$x_3(C_F^*)\le \ptop_3$, which is \eqref{six:thm:ineq:eq6}.
Thus, the inequalities \eqref{six:thm:ineq:eq1} to
\eqref{six:thm:ineq:eq6} are equivalent to the existence
of an elementary path as requested in Proposition \ref{six:prop:phibot1}
when $\ptop_2\neq 1/2$.  The extra inequality
\eqref{six:thm:ineq:eq7} expresses the condition
$\ptop_3\ge x_3(F)$. It is a consequence of the previous
ones because they give $\ptop_3\ge x_3(C^*_F)\ge x_3(F)$.

Finally, if $\ptop_2=1/2$, we have $\pbot_1=0$
because of \eqref{six:thm:ineq:eq1}, thus
$E=B^*=\uf_1$.  Then, for an elementary path
$\cE=AA^*B^*C^*CBA$ to fulfill the requested properties,
we need to have $B=A$ and $C^*=C=F$.
Conversely, for this choice of points, $\cE$ is an
elementary path in $\bar\Delta^{(3)}$. The constraints
$x_1(C)=\pbot_1=0$ and $x_2(C)\le \ptop_2=1/2$ are
automatically satisfied.  The condition $x_2(B)\ge \pbot_2$
reduces to the first inequality in \eqref{six:thm:ineq:eq4}
and implies that $x_2(C^*)\ge \pbot_2$ since
$x_2(C^*)=x_2(F)\ge x_2(A)=x_2(B)$.
Moreover, the condition $x_3(C^*)\le \ptop_3$ reduces to
\eqref{six:thm:ineq:eq7}.  The remaining inequalities
\eqref{six:thm:ineq:eq3}, \eqref{six:thm:ineq:eq5}
and \eqref{six:thm:ineq:eq6} are no restriction
when $\pbot_1=0$ and $\ptop_2=1/2$, while the second inequality
in \eqref{six:thm:ineq:eq4} reduces to $\pbot_2\le 1/3$,
a consequence of \eqref{six:thm:ineq:eq1} together
with the first part of \eqref{six:thm:ineq:eq4}.  This completes
the proof of the first assertion of the theorem.

The proof of the second assertion is dual but there is a slight
complication due to the fact that the boundary of $\bar\Delta^{(3)}$
is not symmetric. The line $x_3=\ptop_3$ must cut this boundary
in a point $D\in\bar{L}$ and a point $D^*\in\bar{L}^*\cup[\uf_1,\uf_3]$.
When $\ptop_3\le 1/2$, we have $D^*\in\bar{L}^*$ and the constraint
$\pbot_2+\ptop_3\le 1$ is automatically satisfied. Otherwise, we need
it in order to ensure that the point $E^*=(1-\pbot_2-\ptop_3,\pbot_2,\ptop_3)$
belongs to $[D,D^*]$.

The final assertion follows immediately thanks
to Corollary \ref{six:cor:spectrum}.
\end{proof}

\begin{proof}[Proof of Theorem \ref{intro:thm:S}]
The above result shows that $\cS$ is defined by polynomial
equalities and inequalities with coefficients in $\bQ$.
Moreover, \eqref{six:thm:ineq:eq1} and
\eqref{six:thm:ineq:eq2} imply that $\cS$ is
contained in $\bR^2\times J\times \bR^2$ where $J$ is the
portion of the curve $(1-2x)(1-2y)=xy$ in $[1/3,1/2]
\times [0,1/3]$.  This proves the first assertion of
the theorem.

Denote by $U$ the largest open subset of $\bR^2\times J\times \bR^2$
contained in $\cS$.  To complete the proof it remains to show that
$U$ is dense in $\cS$.  To this end, fix a point
$\ualpha=(\pbot_1,\pbot_2,\pbot_3,\ptop_1,\ptop_2,\ptop_3)$
of $\cS$ coming from an integral self-similar
$3$-system $\uS$.  Since, by Corollary \ref{self:cor}, such
points are dense in $\cS$, it suffices to show that $\ualpha$
belongs to the closure of $U$.

By Proposition
\ref{semi:prop}, the set $\cK(\uS)$ is the convex hull of a
finite union of strict elementary paths.  Thus the coordinates
of $\ualpha$ satisfy
\[
 0<\pbot_1<\ptop_1<1/3<\pbot_3<\ptop_3<1
 \et
 0<\pbot_2<\ptop_2<1/2.
\]
Consider the elementary path attached to $\ualpha$
by Proposition \ref{six:prop:phibot1}.  Geometric
considerations, based on Figures \ref{six:fig:phibot1:C=E}
and \ref{six:fig:phibot1:B=A}, show that this is a strict
elementary path and that, for each $\epsilon>0$
sufficiently small, and for each
$\delta_1,\delta_2\in (-\epsilon^3,\epsilon^3)$,
it can be deformed into another strict elementary path
$\cE_1=A_1A_1^*B^*_1C^*_1C_1B_1A_1$ of the same type
(with $A_1=B_1$ or $x_2(B^*_1)=x_2(C_1)$), uniquely
determined by the conditions
\[
 x_1(A_1)=\ptop_1+\epsilon,
 \quad
x_2(B^*_1)=\ptop_2+\epsilon^2+\delta_1,
 \quad
 x_1(C_1)=\pbot_1-\epsilon^3+\delta_2,
\]
and that, for this path, we have
\[
  x_2(C_1)\le x_2(B^*_1),
 \quad
 x_2(B_1)>\pbot_2,
 \quad
 x_2(C^*_1)>\pbot_2
 \et
 x_3(C^*_1)<\ptop_3.
\]
It should also be possible to check this directly using the
inequalities of Proposition \ref{six:thm:ineq} (which by hypothesis
hold true for the given choice of $\pbot_1,\dots,\ptop_3$).

Similarly, by deformation of the elementary path provided by
Proposition \ref{six:prop:phitop3}, we find that, for each
$\eta>0$ sufficiently small, and for each
$\delta_3,\delta_4\in\ (-\eta^3,\eta^3)$,
there exists a unique elementary path
$\cE_2=A_2A_2^*B^*_2C^*_2C_2B_2A_2$ of the same type
determined by the conditions
\[
 x_3(A^*_2)=\pbot_3-\eta,
 \quad
 x_2(B_2)=\pbot_2-\eta^2-\delta_3,
 \quad
 x_3(C^*_2)=\ptop_3+\eta^3+\delta_4,
\]
and that it satisfies
\[
 x_2(C^*_2)\ge x_2(B_2),
 \quad
 x_1(C_2)>\pbot_1,
 \quad
 x_2(C_2)<\ptop_2
 \et
 x_2(B^*_2)<\ptop_2.
\]

As the point $(\pbot_3,\ptop_1)$ belongs to the
curve $J$, we can choose $\eta$ so that
$(\pbot_3-\eta,\ptop_1+\epsilon)$ also belongs to
$J$.  Then, we have $A_1=A_2$, $A^*_1=A^*_2$ and, by
Corollary \ref{six:cor:spectrum}, we deduce that
\[
 (\pbot_1-\epsilon^3+\delta_2,\ \pbot_2-\eta^2-\delta_3,\
  \pbot_3-\eta,\ \ptop_1+\epsilon,\
  \ptop_2+\epsilon^2+\delta_1,\ \ptop_3+\eta^3+\delta_4)
\]
belongs to $\cS$. By varying $\epsilon,\delta_1,\dots,
\delta_4$, these points make an open subset of
$\bR^2\times J\times \bR^2$.  So, they are in fact
contained in $U$.  The conclusion follows since
they converge to $\ualpha$ as $\epsilon$ goes to $0$.
\end{proof}



\subsection*{Acknowlegment.}
This research is partially supported by NSERC.


\begin{thebibliography}{99}
%
\bibitem{Ge2012}
  O.~N.~German,
  Intermediate Diophantine exponents and parametric geometry of numbers,
  {\it Acta Arith.\ }{\bf 154} (2012), 79--101.
%
\bibitem{Ja1936a}
  V.~Jarn\'{\i}k,
  \"Uber einen Satz von A.~Khintchine,
  {\it Pr\'ace Mat.-Fiz.\ }{\bf 43} (1935), 151--166.
  %
\bibitem{Ja1936b}
  V.~Jarn\'{\i}k,
  \"Uber einen Satz von A.~Khintchine II,
  {\it Acta Arith.\ }{\bf 2} (1936), 1--22.
%
\bibitem{Ja1938}
  V.~Jarn\'{\i}k,
  Zum Khintchineschen ``\"Ubertragungssatz'',
  {\it Trav.\ Inst.\ Math.\ Tbilissi} {\bf 3} (1938), 193--212.
%
\bibitem{Ja1954}
  V.~Jarn\'{\i}k,
  Contributions \`a la th\'eorie des approximations diophantiennes lin\'eaires et
  homog\`enes, {\it Czechoslovak Math.\ J.\ }{\bf 4} (1954), 330--353
  (in Russian, French summary).
%
\bibitem{K2016}
  A.~Keita,
  On a conjecture of Schmidt for the parametric geometry of numbers,
  {\it Mosc.\ J.\ Comb.\ Number Theory\ }{\bf 6} (2016), 282--292.
%
\bibitem{Kh1926a}
  A.~Y.~Khintchine,
  Zur metrischen Theorie der diophantischen Approximationen,
  {\it Math.\ Z.\ }{\bf 24} (1926), 706--714.
%
\bibitem{Kh1926b}
  A.~Y.~Khintchine, \"Uber eine Klasse linearer diophantischer Approximationen,
  {\it Rend.\ Circ.\ Math.\ Palermo} {\bf 50} (1926), 170--195.
%
\bibitem{La2009}
  M.~Laurent,
  Exponents of Diophantine approximation in dimension two,
  {\it Canad.\ J.\ Math.\ }{\bf 61} (2009), 165--189.
%
\bibitem{La2009b}
 M.~Laurent,
 On transfer inequalities in Diophantine approximation,
  in: Analytic Number Theory in Honour of Klaus Roth,
  Cambridge U.\ Press (2009), 306--314.
%
\bibitem{M2016}
  A.~Marnat,
  About Jarn\'{\i}k's-type relation in higher dimension,
  preprint, 22 pages; arXiv:1510.06334 [math.NT].
%
\bibitem{Mo2012a}
  N.~G.~Moshchevitin,
  Proof of W.~M.~Schmidt's conjecture concerning successive minima of a lattice,
   {\it J.\ Lond.\ Math.\ Soc.\ }{\bf 86} (2012), 129--151.
%
\bibitem{Mo2012b}
  N.~G.~Moshchevitin,
  Exponents for three-dimensional simultaneous Diophantine approximations,
  {\it Czechoslovak Math.\ J.\ }{\bf 62} (2012), 127--137.
%
\bibitem{R2015}
  D.~Roy,
  On Schmidt and Summerer parametric geometry of numbers,
  {\it Ann.\ of Math.\ }{\bf 182} (2015), 739--786.
%
\bibitem{R2016}
  D.~Roy,
  Spectrum of the exponents of best rational approximation,
  {\it Math.\ Z.\ }{\bf 283} (2016), 143--155.
%
\bibitem{Sc1967}
  W.~M.~Schmidt,
  On heights of algebraic subspaces and diophantine approximations,
  {\it Ann.\ of Math.\ } {\bf 85} (1967), 430--472.
%
\bibitem{Sc1982}
  W.~M.~Schmidt,
  Open problems in Diophantine approximations,
  in: {\it Approximations diophantiennes et nombres transcendants}
  (Luminy 1982), Progr.\ Math., vol.~31, pp. 271--287,
  Birkh\"auser, Boston, 1983.
%
\bibitem{SS2009}
  W.~M.~Schmidt and L.~Summerer,
  Parametric geometry of numbers and applications,
  \textit{Acta Arith.\ }\textbf{140} (2009), 67--91.
%
\bibitem{SS2013a}
  W.~M.~Schmidt and L.~Summerer,
  Diophantine approximation and parametric geometry of numbers,
  \textit{Monatsh.\ Math.\ }\textbf{169} (2013), 51--104.
%
\bibitem{SS2013b}
  W.~M.~Schmidt and L.~Summerer,
  Simultaneous approximation to three numbers,
  \textit{Mosc.\ J.\ Comb.\ Number Theory} \textbf{3} (2013), 84--107.
%
\bibitem{SS2016a}
  W.~M.~Schmidt and L.~Summerer,
  The generalization of Jarnik's identity,
  \textit{Acta Arith.\ }\textbf{175} (2016), 119--136.
%
\bibitem{SS2016b}
  W.~M.~Schmidt and L.~Summerer,
  Simultaneous approximation to two reals: bounds for the second successive minimum,
  Mathematika (to appear in a special issue in memory of Klaus Roth), 14 pages.
%
\bibitem{Se1954}
  A.~Seidenberg,
  A new decision method for elementary algebra,
  \textit{Ann.\ of Math.\ }\textbf{60} (1954), 365--374.
\end{thebibliography}
\end{document}